\documentclass[12pt]{siamart190516}
\usepackage[margin = 3cm]{geometry}
\usepackage[utf8]{inputenc} 
\usepackage[T1]{fontenc}    
\usepackage{url}            
\usepackage{hyperref}
\usepackage{booktabs}       
\usepackage{amsfonts}       
\usepackage{epstopdf}
\usepackage{nicefrac}       
\usepackage{microtype}      
\usepackage{lipsum}
\usepackage{amsmath,amssymb,mathrsfs}
\usepackage{bbm}
\usepackage{tcolorbox}
\usepackage{enumitem}
\usepackage{mathtools}
\usepackage{latexsym}
\usepackage{multimedia}
\usepackage{media9}
\usepackage{relsize}
\usepackage{algorithm}
\usepackage{algorithmic}
\usepackage{lineno}
\usepackage{tikz}
\usetikzlibrary{arrows.meta}
\usepackage{comment}
\usepackage{multicol}
\usepackage{graphicx}
\usepackage{subfig}
\usepackage{xcolor}
\usepackage{cleveref}
\usepackage{todonotes}
\usepackage{setspace}

\DeclareMathOperator*{\argmin}{argmin}

\newcommand{\bq}{\begin{equation}}
\newcommand{\eq}{\end{equation}}
\newcommand{\R}{\mathbb{R}}

\newcommand{\Prob}{\mathbb{P}}

\newcommand{\red}[1]{{\color{red}{#1}}}

\newcommand*\Laplace{\mathop{}\!\mathbin\bigtriangleup}

\newcommand{\dlim}{\displaystyle\lim}
\newcommand{\dint}{\displaystyle\int}

\newcommand{\eps}{\varepsilon}

\newcommand{\bbP}{\mathbb P} 
\newcommand{\bbR}{\mathbb R} \newcommand{\bbS}{\mathbb S}

\newcommand{\bzeta}{\boldsymbol \zeta}

\newcommand{\be}{\mathbf e}

 \newcommand{\bx}{\mathbf x} 
  
\newcommand{\bA}{\mathbf A}

\newcommand{\cA}{\mathcal A} \newcommand{\cB}{\mathcal B}
\newcommand{\cC}{\mathcal C}

 \newcommand{\cL}{\mathcal L}

\newcommand{\wt}{\widetilde}
\newcommand{\wh}{\widehat}

\newcommand*\patchAmsMathEnvironmentForLineno[1]{%
\expandafter\let\csname old#1\expandafter\endcsname\csname #1\endcsname
\expandafter\let\csname oldend#1\expandafter\endcsname\csname end#1\endcsname
\renewenvironment{#1}%
{\linenomath\csname old#1\endcsname}%
{\csname oldend#1\endcsname\endlinenomath}}%
\newcommand*\patchBothAmsMathEnvironmentsForLineno[1]{%
\patchAmsMathEnvironmentForLineno{#1}%
\patchAmsMathEnvironmentForLineno{#1*}}%
\patchBothAmsMathEnvironmentsForLineno{equation}%
\patchBothAmsMathEnvironmentsForLineno{align}%
\patchBothAmsMathEnvironmentsForLineno{flalign}%
\patchBothAmsMathEnvironmentsForLineno{alignat}%
\patchBothAmsMathEnvironmentsForLineno{gather}%
\patchBothAmsMathEnvironmentsForLineno{multline}%

\newsiamremark{remark}{Remark}

\graphicspath{{Figures/}}

\ifpdf
  \DeclareGraphicsExtensions{.eps,.pdf,.png,.jpg}
\else
  \DeclareGraphicsExtensions{.eps}
\fi

\newcommand{\TheShortTitle}{AdaVar for Global Optimization}
\newcommand{\TheAuthors}{B.\ Engquist, K.\ Ren and Y.\ Yang}

\headers{\TheShortTitle}{\TheAuthors}

\title{Adaptive State-Dependent Diffusion for derivative-free optimization}
\author{
  Bj\"orn Engquist\thanks{Department of Mathematics and the Oden Institute, The University of Texas, Austin, TX 78712;
  \texttt{engquist@oden.utexas.edu}} 
   \and Kui Ren\thanks{Department of Applied Physics and Applied Mathematics, Columbia University, New York, NY 10027; \texttt{kr2002@columbia.edu}}
  \and Yunan Yang\thanks{Institute for Theoretical Studies, ETH Z\"urich, Z\"urich, Switzerland 8092; 
 \texttt{yunan.yang@eth-its.ethz.ch}} 
}
\ifpdf
\hypersetup{
  pdftitle={\TheShortTitle},
  pdfauthor={\TheAuthors}
}
\fi

\begin{document}
\maketitle

\tableofcontents

\begin{abstract}
This paper develops and analyzes a stochastic derivative-free optimization strategy. A key
feature is the state-dependent adaptive variance. We prove global convergence in probability with algebraic rate and give the quantitative results in numerical examples. A striking fact is that convergence is achieved without explicit information of the gradient and even without comparing different objective function values as in established methods such as the simplex method and simulated annealing. It can otherwise be compared to annealing with state-dependent temperature.
\end{abstract}

\begin{keywords}
derivative-free optimization, global optimization, adaptive diffusion, stationary distribution, Fokker--Planck theory
\end{keywords}

\begin{AMS}
90C26, 90C15, 65K05
\end{AMS}


\section{Introduction}

The idea of using randomness to achieve global convergence in numerical optimization algorithms has been extensively explored. Different stochastic mechanisms have been developed in the literature based on time-dependent diffusion~\cite{GeHw-SIAM86,chiang1987diffusion,henderson2003theory,chow2013global,FoKlRi-arXiv21,frederick2022collective}. In~\cite{EnReYa-arXiv22}, we introduced a stochastic gradient descent method for global optimization with a time- and state-dependent variance. Through rigorous analysis of the discrete algorithm and several numerical examples, we demonstrated the global convergence of the algorithm under mild assumptions on the objective function. In this paper, we improve the result in~\cite{EnReYa-arXiv22} by considering a derivative-free version of the algorithm. We will prove that the new algorithm can still achieve global convergence by using a single particle performing Brownian motion where the diffusion coefficient is monotone with respect to the objective function.

To describe the algorithm, let $\Omega\subset\bbR^d$ ($d\ge 1$) be a smooth bounded domain and $f(\bx): \Omega \mapsto \bbR$ a sufficiently regular objective function. We are interested in finding the global minima of $f$ using iterative schemes of the form
\begin{equation}\label{EQ:Adaptive A}
    X_{n+1}=X_n + \sqrt{\eta}\, \sigma\big(f(X_n)\big)\, \bzeta_n,\ \ \ n\ge 0
\end{equation}
where $\{\bzeta_n\}_{n\ge 0}$ are i.i.d.~standard normal random vectors, $\eta>0$ is the step size, and $\sigma$ controls the variance of the randomness.  Following our previous work~\cite{EnReYa-arXiv22}, we consider adaptive schemes for selecting function-dependent $\sigma$ values for the iteration. To mimic the classical diffusion setup~\cite{GeHw-SIAM86}, and also to regularize the degeneracy as done in the literature~\cite{Oleinik-MS66,StVa-Book97}, we introduce a regularization $\eps(t)>0$ with the property that $\eps(t)\to 0$ as $t\to\infty$, and define the regularized diffusion coefficient $\sigma = \sigma_\eps$ as
\begin{equation}\label{EQ:Sigma}
    \sigma_\eps (f)=\sqrt{2\Big[ \left(f(\bx)-f_{\min}^* \right)^{+}\Big]^\beta +\eps(t) }\,,
\end{equation}
where the exponent $\beta\ge d/2$ ($d$ being the dimension of the underlying space), $a^+:=\max(a, 0)$, and $f_{\min}^*$ is an approximation to $f_{\min}$, the minimum value of the function $f(\bx)$ on $\Omega$ defined by $f_{\min}:=\min_{\bx\in\overline\Omega} f(\bx)$. When $f_{\min}$ is known \emph{a priori}, we take $f_{\min}^*=f_{\min}$. When $f_{\min}$ is not known, we select $f_{\min}^*$ in other ways (which we will describe in more detail later in the section on numerical simulations). For instance, one choice that has been explored a little bit in~\cite{EnReYa-arXiv22} and will be investigated more later is the case when $f_{\min}^*$ is taken as the minimum value of $f$ in part of the history of the iteration. That is,
\begin{equation}
    f_{\min}^*:=\min_{n-1-m\le k\le n-1} f(X_k),\ \ 1\le m\le n-1\,.
\end{equation}
Without loss of generality, we assume that $\Omega$ is a $d$-dimensional cube with edge length $\ell_\Omega$, and consider the iteration with periodic boundary condition
\begin{equation}\label{EQ:Periodicity}
    X_{n}+\ell_\Omega \be_i=X_{n}, \ \ \ \ \forall n\ge 0,\ \ 1\le i\le d,
\end{equation}
with $\be_i$ being the unit vector in direction $i$. We assume that $f$ is periodically extended to $\bbR^d$ to satisfy $f(\bx+\ell_\Omega \be_i)=f(\bx)$, $\forall 1\le i\le d$.

The scheme~\eqref{EQ:Adaptive A} is a derivative-free stochastic iteration, as it does not explicitly involve the derivative of the objective function $f$. In the rest of this work, we will show that algorithm~\eqref{EQ:Adaptive A}, with appropriately selected $\eps(t)$, on a continuous level and under reasonable assumptions, can be globally convergent with an algebraic rate. To be more specific, we show a probability result of the form: 
\[
 \bbP \left(|X_t - \bx_*| > t^{-\nu} \right)   \lesssim t^{-\kappa'}, \quad \beta>d/2,
\]
for some $\nu,\kappa'>0$; see more details in Theorem~\ref{thm:main} and Corollary~\ref{coro:convergence in p}.
We will also provide some numerical examples in applications to show its practical relevance; see Sections~\ref{sec:numerics} and~\ref{SEC:Gen}. Moreover, while our primary focus is to study the derivative-free algorithm~\eqref{EQ:Adaptive A}, we will see that adding explicit gradient information to the algorithm will significantly accelerate its convergence; see Figure~\ref{fig:2D gradient comp}.

There are many effective derivative-free methods in the literature for optimization~\cite{CaRo-MP22,CoScVi-Book09,KoLeTo-SIAM03,LaMeWi-AN19}. While it is impossible to have an xhaustive list of successive methods in this direction, let us mention, as examples, the Nelder--Mead (NM) method~\cite{NeMe-CJ65,McKinnon-SIAM99,LaReWrWr-SIAM98}, which performs a direct search in the parameter space using function value comparison, the genetic algorithm (GA)~\cite{HaHa-Book04,Mitchell-Book98,Holland-SA92}, simulated annealing (SA)~\cite{chiang1987diffusion,GeHw-SIAM86,dekkers1991global,henderson2003theory,kirkpatrick1983optimization}, the particle swarm optimization (PSO) method~\cite{PoKeBl-SI07,ShEb-IEEE98}, and the consensus-based optimization method~\cite{FoKlRi-arXiv21,totzeck2021trends}. Different variations of such methods have been proposed to solve problems with different features. Interested readers are referred to~\cite{AlAuGhKoLe-EJCO21,CoScVi-Book09,KoLeTo-SIAM03} and references therein for an overview of some of the recent developments in the field. Let us emphasize that some of the methods mentioned above aim at local optimization, and most of them include gradient information implicitly, for instance, by utilizing the difference of objective function values at two different points of the parameter space in the design of the algorithms. However, the scheme~\eqref{EQ:Adaptive A} does not involve such gradient information in its form. It can be seen as a variant of the Brownian motion where the diffusion coefficient is chosen to depend on the current value of the objective function.



The rest of the paper is structured as follows. In~\Cref{SEC:Theory}, we present a convergence theory for the algorithm in the continuous limit. We use a regularized version of~\eqref{EQ:Sigma} and assume the value of the global minimum of the objective function is known. In~\Cref{sec:numerics}, we provide numerical simulations to validate this convergence result. We discuss the algorithm in more practical settings in~\Cref{SEC:Gen} for cases where the gradient information can be added and the objective function value at the global minimum is unknown a priori. We also point out in~\Cref{SEC:compare} the close connection as well as main differences of scheme~\eqref{EQ:Adaptive A} to our previous work of~\cite{EnReYa-arXiv22}. Concluding remarks are presented in~\Cref{SEC:Concl}.

\section{Asymptotic Behavior via the Fokker--Planck Equation}
\label{SEC:Theory}

We are interested in obtaining a systematic understanding of the algorithm~\eqref{EQ:Adaptive A}.  As a starting point, we will analyze this iterative scheme in the continuous limit (whose existence we formally assume, for instance, when $\eta\to 0$ at a proper rate). The iteration is described by the stochastic differential equation (SDE)
\begin{equation}\label{EQ:Adaptive A Cont}
    dX_t = \sigma(f)\, d W_t\,,
\end{equation}
where $\{W_t\}_{t\ge 0}$ is a standard $d$-dimensional Brownian motion. We formally introduce the generator $\cL$ of the process $\{X_t\}_{t\ge 0}$ as
\begin{equation}\label{eq:generator}
    \cL \rho :=\frac{1}{2}\sigma^2 \Delta\rho,\ \ \rho\in \cC_{\rm per}^2(\bbR^d)\,,
\end{equation}
where $\Delta$ is the  standard Laplacian operator in dimension $d$ and the subscript ``${\rm per}$'' in $\cC_{\rm per}^2(\bbR^d)$ is used to reflect the fact that functions in the space are $\Omega$-periodic. Then the Fokker--Planck equation for the distribution $u(x, t)$ of the process is of the form
\begin{equation}\label{EQ:Fokker--Planck}
    \partial_t u = \cL^* u:=\frac{1}{2}\Delta(\sigma^2 u\big),\quad \mbox{in}\ \ \bbR^d \times(0, +\infty),\ \ \ u(\bx, 0)=u_0,\quad \mbox{in}\ \ \bbR^d\,,
\end{equation}
assuming that the initial distribution we started the process with, $u_0$, is also $\Omega$-periodic.

The fact that $\sigma(f_{\min}^*)=0$ means that the SDE~\eqref{EQ:Adaptive A Cont}, as well as the PDE~\eqref{EQ:Fokker--Planck}, are \textit{degenerate}. Moreover, as we will see later, our assumption on $f(\bx)$ and our selection of $\beta\ge d/2$ allow singular measures to be admissible solutions to the Fokker--Planck equation~\eqref{EQ:Fokker--Planck}. These factors make it nontrivial to fully characterize the behavior of the process $\{X_t\}_{t\ge 0}$.

We denote by $\cL_\eps$ the generator associated with the process with $\sigma_\eps$ in~\eqref{EQ:Sigma}. That is,
\begin{equation}\label{EQ:Deps}
    \cL_\eps :=D_\eps \Delta, \qquad D_\eps:=\frac{1}{2} \sigma_\eps^2 = \Big((f(\bx)-f_{\min}^*)^+\Big)^\beta+\eps\,.
\end{equation}
We will see later through numerical simulations that the algorithm~\eqref{EQ:Adaptive A} with adaptive diffusion~\eqref{EQ:Sigma} can be quite efficient in general. 

While iteration~\eqref{EQ:Adaptive A} is derivative-free in nature as it does not explicitly have the gradient of the objective function involved, gradient information is indeed encoded in the algorithm. This can be seen on the heuristic level from the Fokker--Planck equation~\eqref{EQ:Fokker--Planck}. Indeed, after a little rearrangement, the equation can be written as
\begin{equation} \label{eq:grad-FK}
    \partial_t u = \nabla\cdot(u\, \nabla \sigma) +\frac{1}{2}\sigma(f)\, \Delta u-\frac{1}{2}(\Delta \sigma) u\,.  
\end{equation}
At a given function value $f$, the first three terms of this Fokker--Planck equation correspond to the stochastic differential equation~\eqref{EQ:Adaptive A Cont} with an additional drift term $-\nabla\sigma\, dt$ on the right-hand side. The drift term $\nabla\sigma=\sigma'(f)\nabla f$ (and $\sigma'(f)>0$ under the assumptions) clearly depends on the gradient of the objective function. The last term, $-\frac{1}{2}(\Delta \sigma) u$, adds an absorption/generation mechanism in the process at locations where $\sigma$ is convex/concave.

We first provide some theoretical investigations of our algorithm in the case where the value of the global minimum of $f$, denoted by $f_{\min}$, is known \emph{a priori}. In this case, we take $f_{\min}^*=f_{\min}$ in~\eqref{EQ:Sigma}. 
We make the following assumptions on the objective function $f$.
\begin{enumerate}[label=\textbf{A\arabic*}]
\item \label{itm:A1} The function $f(\bx)$ is at least $\cC^2$ and is $\Omega$-periodic with a unique global minimizer $\bx_*\in \Omega$ with $f_{\min}:=f(\bx_*)$. Moreover, there is a gap $\mathfrak g$ between the global minimum value $f_{\min}$ and other local minima of $f(\bx)$.
\item \label{itm:A2} There exists $r,\, a > 0$ such that $f(\bx)-f_{\min} \leq  a |\bx-\bx_*|^2$ on $\cB_r(\bx_*) := \{\bx \in \mathbb{R}^d: |\bx-\bx_*| <  r\}\subset \Omega$. 
\item \label{itm:A3} There exists $b>0$ such that $f(\bx)-f_{\min}  \geq  b |\bx-\bx_*|^2$ for all $\bx\in\Omega$.
\end{enumerate}
\begin{remark}
    The rationale for making some of the assumptions in~\ref{itm:A1}-\ref{itm:A3} is mainly to simplify the presentation, as it will be evident from the discussions in the rest of this section that these assumptions can be relaxed significantly for the main results to remain valid. For example, the theoretical result will hold if we replace~\ref{itm:A1}-\ref{itm:A3} with
    \begin{enumerate}[label=\textbf{B\arabic*}]
    \item \label{itm:B1} The function $f(\bx)$ is $\Omega$-periodic with $K<+\infty$ global minimizers $\{\bx_k\} \subset \Omega$ with $f_{\min}:=f(\bx_k)$ $\forall 1\le k\le K$. Moreover, there is a gap $\mathfrak g$ between the global minimum value $f_{\min}$ and other local minima of $f(\bx)$.
    \item \label{itm:B2} There exists $r,\, a > 0$ such that for each $1\le k\le K$, $\Big(f(\bx)-f_{\min}\Big)^\beta \leq  a |\bx-\bx_k|^{d_*}$ on $\cB_r(\bx_k) := \{\bx \in \mathbb{R}^d: |\bx-\bx_k| <  r\}\subset \Omega$ for some $d_*\ge d$.
    \item \label{itm:B3} There exists $b>0$ such that for each $1\le k\le K$, $\Big(f(\bx)-f_{\min}\Big)^\beta \geq  b |\bx-\bx_*|^{d_*}$ for all $\bx\in\Omega$ and some $d_*\ge d$.
    \end{enumerate}
    In particular, \ref{itm:B2} and~\ref{itm:B3} say that the behavior of  $(f-f_{\min})^\beta$ is an essential component of the analysis.
    This makes the theory work for a larger class of objective functions. We will provide numerical simulations in~\Cref{sec:numerics} to illustrate the case with multiple global minimizers. 
\end{remark}

Our main results will be based on the analysis of the Fokker--Planck equation associated with the generator $\cL_\eps$ with regularization of the form
\begin{equation}\label{EQ:Epst}
    \eps(t)=(1+t)^{-\alpha},\ \ t\ge 0\,,
\end{equation}
for some $\alpha>0$. That is,
\begin{equation} \label{eq:eqn_u}
\partial_t u = \Delta \left( D_\eps u\right), \ \ \mbox{in}\ \ \bbR^d\times(0, +\infty), \ \ \ u(\bx, 0)=u_0(\bx), \ \ \mbox{in}\ \ \bbR^d\,,
\end{equation}
where all quantities involved are $\Omega$-periodic, and $D_\eps$ is defined in~\eqref{EQ:Deps}. We are interested in  solutions representing probability distributions, so we additionally require the normalization condition
\[
    \int_{\Omega}u(\bx, t) d\bx =1,\ \ \forall t\ge 0\,.
\]



The main strategy for constructing a solution to~\eqref{eq:eqn_u} is based on the instantaneous equilibrium distribution of the problem with a fix $\eps(t^*)$ for some $t^*$. For that purpose, for any given $t>0$, we denote by $\bar u(\bx, t)$ an $\Omega$-periodic function that solves
\begin{equation} \label{eq:eqn_u_bar}
\Delta \left( D_\eps(\bx,t) \bar{u}(\bx,t) \right) = 0\,,
\end{equation}
with the normalization condition $\int_\Omega \bar{u}(\bx,t) d\bx = 1$.


With these assumptions, we can prove the following results. 
\begin{theorem}\label{thm:main}
Under assumptions~\ref{itm:A1}-\ref{itm:A3}, let $u$ and $\bar{u}$ be solutions to~\eqref{eq:eqn_u} and \eqref{eq:eqn_u_bar}, respectively, for $\eps(t)$ given in~\eqref{EQ:Epst}. Take $\beta \geq \frac{d}{2}$ and $\alpha \in \left(0,\frac{1}{2} \right] \cap \left(0,  \frac{2\beta }{d + 3\beta}\right)$. Then there exists $t_0>0$ such that for all $t>t_0$ we have
\begin{equation}\label{eq:gamma}
   \big \| u(x,t) - \bar{u}(x,t) \big \|_{L^2\left(\mu \right)}   \lesssim   t^{- \gamma},\qquad \gamma = 1 - \left(\frac{d}{2\beta} + \frac{3}{2} \right)\alpha  > 0\,,
\end{equation}
where $\|\cdot \|_{L^2(\mu)}$ denotes the weighted $L^2$ norm with measure $d \mu = D_\eps(\bx,t)\, d\bx $.
\end{theorem}

Theorem~\ref{thm:main} yields the following corollary which states that the process $\{X_t\}_{t\ge 0}$ generated by~\eqref{EQ:Adaptive A Cont} with $\sigma_\eps$ given in~\eqref{EQ:Sigma} and $\eps$ given in~\eqref{EQ:Epst} converges in probability to the global minimizer $\bx_*$ of $f(\bx)$.
\begin{corollary}\label{coro:convergence in p}
Let $\beta > d/2$. Then, under the same setting as in Theorem~\ref{thm:main}, for any $\delta>0$, we have that
\begin{equation}\label{eq:kappa}
    \bbP \left(|X_t - \bx_*| > \delta \right)   \lesssim t^{-\kappa}, \quad  \kappa = \min \left(\gamma,  \left( 1 - \frac{d}{2\beta} \right)\alpha \right)\,,
\end{equation}
for all $t>t_0$. Moreover, if we take $\delta=t^{-\nu}$ with $\nu$ such that $0<\nu<\min(\gamma, (\frac{1}{2}-\frac{d}{4\beta})\alpha)$, then we have
\begin{equation}\label{eq:kappa2}
    \bbP \left(|X_t - \bx_*| > t^{-\nu} \right)   \lesssim t^{-\kappa'}, \quad  \kappa' = \min \left(\gamma-\nu,  \left( 1 - \frac{d}{2\beta} \right)\alpha-2\nu \right)\,,
\end{equation}
for all $t>t_0$.
\end{corollary}
\begin{remark}
When $\beta = d/2$, we obtain the standard  logarithmic convergence as $ \Prob \left(|X_t - \bx_*| > \delta \right)   \lesssim  (\log t)^{-1}$ after applying Theorem~\ref{thm:main}.
\end{remark}
The rest of this section is devoted to the proof of these results.

\subsection{Preliminaries in the Case of Fixed \texorpdfstring{$\eps$}{}}

The solution $\bar u$ of~\eqref{eq:eqn_u_bar}, which we refer to as the instantaneous equilibrium distribution, is the equilibrium solution for the problem~\eqref{eq:eqn_u} with a fixed $\eps>0$. 
It is straightforward to verify that, when $D_\eps^{-1} \in L^1(\Omega)$, $\bar u$ is given as
\begin{equation}\label{EQ:Stationary Dist}
    \bar u(\bx) =  Z_{\bar u}^{-1} D_\eps^{-1}=Z_{\bar u}^{-1}\frac{1}{\Big(f(\bx)-f_{\min}\Big)^\beta+\eps},\qquad Z_{\bar u}:=\|D_\eps^{-1}\|_{L^1(\Omega)}\,.
\end{equation}
Note that periodicity and non-negativity of $\bar u(\bx)$ force out solutions of the form $D_\eps^{-1}(\bA\cdot \bx+B)$ for some vector $\bA$ and constant $B$. 

We first show that $\bar u(\bx)$ is well-defined for any fixed $\eps>0$. This requires us to show that $Z_{\bar u}$ is finite, in which case $\bar u(\bx)\ge 0$ and $\dint_{\Omega} \bar u(\bx) d\bx=1$. We have the following lemma.
\begin{lemma}\label{LMMA:Well}
    Under assumptions~\ref{itm:A1}-\ref{itm:A3}, for any given $\eps>0$, we have that
    \begin{equation*}
        0< Z_{\bar u} \le \frac{5}{2}V_{\Omega}\, \eps^{-1}
    \end{equation*}
   with $V_\Omega$ the volume of $\Omega$. Moreover, when $\eps$ is sufficiently small, we have that
    \begin{equation}\label{EQ:Lower Bounds}
        Z_{\bar u}\ge \left\{
        \begin{matrix}
            C_1\, \eps^{-\frac{2\beta-d}{2\beta}},& \beta>d/2\\
            C_2\, \log \left( 1/\eps \right), & \beta=d/2
        \end{matrix}
        \right.
    \end{equation}
    for some positive constants $C_1$ and $C_2$ independent of $\eps$.
\end{lemma}
\begin{proof}
    Let $r$, $\mathfrak g$, $a$ and $b$ be defined as in the assumptions~\ref{itm:A1}-\ref{itm:A3}, and define $\phi:=(f-f_{\min})^\beta$. We first derive the upper bound to show that $\bar u(\bx)$ is well-defined. 
    We observe first, using the notation $\phi_{\le \eps}:=\{\bx : \phi(\bx)\le \eps\}$ and $\phi^c_{\le \eps}:=\left(\{\bx : \phi(\bx)\le \eps\}\right)^c = \{\bx : \phi(\bx)>\eps\} $, that
    \begin{align}
       Z_{\bar u}=\int_{\Omega} \frac{1}{\phi+\eps} d\bx &=\int_{\cB_r(\bx_*)} \frac{1}{\phi+\eps} d\bx +\int_{\cB_r(\bx_*)^c} \frac{1}{\phi+\eps} d\bx \nonumber \\ 
       &=\int_{\cB_r(\bx_*)} \frac{1}{\phi+\eps} d\bx +\int_{\cB_r(\bx_*)^c\cap \phi_{\le \eps}} \frac{1}{\phi+\eps} d\bx  +\int_{\cB_r(\bx_*)^c\cap \phi_{\le \eps}^c} \frac{1}{\phi+\eps} d\bx\,. \label{EQ:Decomp}
    \end{align}
    The first term is bounded by $\eps^{-1} V_{\cB_r(\bx_*)}$ with $V_{\cB_r(\bx_*)}$ the volume of the ball $\cB_r(\bx_*)$. By the assumption on $f(\bx)$, the set $\phi_{\le \eps}$ is compact. Therefore, the second term is bounded by $\eps^{-1} V_{\cB_r(\bx_*)^c\cap \phi_{\le \eps}}\le \eps^{-1}V_{\Omega}$. To bound the last term, we use the assumptions on $f$ to get
    \[
    \int_{\cB_r(\bx_*)^c\cap \phi_{\le \eps}^c} \frac{1}{\phi+\eps} d\bx\le \int_{\cB_r(\bx_*)^c\cap \phi_{\le \eps}^c} \frac{1}{2\eps} d\bx\le \frac{V_{\Omega}}{2\eps}\,.
    \]
    We can now combine the three terms to get the upper bound of $Z_{\bar u}$. 
    
    To derive the lower bounds in~\eqref{EQ:Lower Bounds}, we assume that $\eps<f_{\min}+\mathfrak g$. We observe from~\eqref{EQ:Decomp} that
    \begin{align*}
        Z_{\bar u} &\ge \int_{\cB_r(\bx_*)} \frac{1}{\phi+\eps} d\bx \\
        & =\int_{\cB_r(\bx_*)\cap \phi_{\le\eps}} \frac{1}{\phi+\eps} d\bx+\int_{\cB_r(\bx_*)\cap \phi_{\le\eps}^c} \frac{1}{\phi+\eps} d\bx 
         \ge \frac{1}{2\eps}\int_{\cB_r(\bx_*)(r)\cap \phi_{\le\eps}} d\bx +\frac{1}{2}\int_{\cB_r(\bx_*)\cap \phi_{\le\eps}^c} \frac{1}{\phi} d\bx\,.
    \end{align*}
    Let $r_\eps$ be such that 
    \[
        a^\beta r_\eps^{2\beta}=\eps,\ \ \ \mbox{or equivalently},\ \ \ r_\eps=a^{-1/2}\eps^{1/2\beta}\,.
    \]
    Then, we have,  by denoting $c_d=\frac{\pi^{d/2}}{\Gamma(\frac{d}{2}+1)}$ (the volume of the unit ball in $\bbR^d$), that 
    \begin{equation*}
        \int_{\cB_r(\bx_*)\cap \phi_{\le\eps}} d\bx=\left\{
        \begin{matrix}
            V_{\cB_r(\bx_*)} &=&c_d\, r^d, && & r_\eps\ge r,\\
            V_{\phi_{\le\eps}}&=&c_d\, r_\eps^d &=&c_d\, a^{-d/2}\eps^{d/2\beta}, &\, r_\eps< r\,.\\
        \end{matrix}
        \right.
    \end{equation*}
    Moreover, when $r_\eps<r$, we have that
    \begin{multline*}
        \int_{\cB_r(\bx_*)\cap \phi_{\le\eps}^c} \frac{1}{\phi} d\bx = \int_{\cB_r(\bx_*)\cap  \cB_{r_\eps}(\bx_*)^c} \frac{1}{\phi} d\bx \ge \int_{\cB_r(\bx_*)\cap \cB_{r_\eps}(\bx_*)^c} \frac{1}{a^\beta|\bx-\bx_*|^{2\beta}} d\bx\\
        =\frac{\cA(\bbS^{d-1})}{a^\beta} \int_{r_\eps}^r s^{d-1-2\beta} ds =\frac{\cA(\bbS^{d-1})}{a^\beta} \left\{
        \begin{matrix}
            \frac{1}{d-2\beta}(r^{d-2\beta}-r_{\eps}^{d-2\beta}), & \mbox{when}\ \beta>d/2\\
            \log r/r_{\eps}, & \,\mbox{when}\ \beta=d/2\,.
        \end{matrix}
        \right.
    \end{multline*}
    Here, $\cA(\bbS^{d-1})$ denotes the area of the sphere $\bbS^{d-1}$.
    We can now put these bounds together and utilize the fact that $\dint_{\cB_r(\bx_*)\cap \phi_{\le\eps}^c} \frac{1}{\phi} d\bx>0$ when $r_\eps\ge r$, to finish the proof.
\end{proof}
The above calculation shows that $Z_{\bar u}$, the integral of $D_\eps^{-1}$ over $\Omega$, blows up as $\eps\to 0$. This is a key feature needed for the distribution $\bar u$ to concentrate on the global minimizer $\bx_*$ for sufficiently small $\eps$, as we prove in the next lemma.
\begin{lemma}\label{LMMA:Concentration}
    Under assumptions~\ref{itm:A1}-\ref{itm:A3}, for any given function value $\mathfrak f>f_{\min}$ and $\delta>0$, there exists $\eps_0>0$ such that for any $\eps\le \eps_0$,
    \[
        \int_{\{\bx: f(\bx)\le \mathfrak f\}} \bar u(\bx)\, d\bx \ge 1-\delta\,,
    \]
    where $\bar u$, depending on $\eps$, is defined in~\eqref{EQ:Stationary Dist}.
\end{lemma}
\begin{proof}
With the same notation $\phi(\bx):=(f(\bx)-f_{\min})^\beta$, we observe that
\[
    \int_{\{\bx: f(\bx)\le \mathfrak f\}} \bar u(\bx)\, d\bx =\frac{1}{Z_{\bar u}} \int_{\{\bx: f(\bx) \le \mathfrak f\}} \frac{1}{\phi +\eps} d\bx=1-\frac{1}{Z_{\bar u}} \int_{\{\bx: f(\bx) > \mathfrak f\}} \frac{1}{\phi+\eps} d\bx\,.
\]
Meanwhile, it is straightforward to see that 
\begin{align*}
    \frac{1}{Z_{\bar u}} \int_{\{\bx: f(\bx) > \mathfrak f\}} \frac{1}{\phi+\eps} d\bx   & \le \frac{1}{Z_{\bar u}} \int_{\{\bx: f(\bx) > \mathfrak f\}} \frac{1}{\phi(\bx)} d\bx \\ & \le \frac{1}{Z_{\bar u}} \int_{\{\bx: f(\bx) > \mathfrak f\}} \frac{1}{(\mathfrak f-f_{\min})^\beta} d\bx  \le \frac{1}{Z_{\bar u}} \frac{V_\Omega}{(\mathfrak f-f_{\min})^\beta}\,.
\end{align*}
By the result of Lemma~\ref{LMMA:Well}, we have that
\begin{equation}\label{EQ:Probability Bound}
    \frac{1}{Z_{\bar u}} \frac{V_\Omega}{(\mathfrak f-f_{\min})^\beta} \le \frac{V_\Omega}{(\mathfrak f-f_{\min})^\beta}  \left\{
    \begin{matrix}
        \frac{1}{C_1}\, \eps^{\frac{2\beta-d}{2\beta}},& \beta>d/2\\
        \frac{1}{C_2}\, \Big(\log \frac{1}{\eps}\Big)^{-1}, & \beta=d/2
    \end{matrix}
    \right.
\end{equation}
with $C_1$ and $C_2$ given as in~\eqref{EQ:Lower Bounds}. It is then clear that we can select $\eps=\eps_0$ small enough to make this term smaller than $\delta$. The rest then follows from the monotonicity of the bound in~\eqref{EQ:Probability Bound} with respect to $\eps$.
\end{proof}
As we can see, \Cref{coro:convergence in p}, based on a finely-tuned time-dependent $\eps$, provides a more precise characterization of this result. It says that when we tune $\eps(t)$ at the rate of $t^{-\alpha}$, we get that the rate of concentration is $\lesssim t^{-\gamma}$ for some $\gamma>0$ depending on $\alpha$.

The calculation in~\Cref{LMMA:Concentration} also suggests that the distribution $\bar u$ converges to the delta-measure at $\bx_*$. This is indeed the case as we see from the following lemma.
\begin{lemma}
    Under assumptions~\ref{itm:A1}-\ref{itm:A3}, we have that $\bar u(\bx) \to \delta(\bx-\bx_*)$ weakly as $\eps\to 0$.
\end{lemma}
\begin{proof}
Let $\psi(\bx)\in\cC^\infty(\overline\Omega)$ be a given function. We first assume $\beta>d/2$. Define $\zeta:=\frac{2\beta-d}{2\beta}$ and take $\eta\in(0, \zeta)$. We then have, using the same decomposition as in the proof of the previous lemma, that
\begin{equation}\label{EQ:Decom Delta Proof}
   \int_\Omega \psi(\bx) \bar u(\bx) d\bx-\psi(\bx_*) = \frac{1}{Z_{\bar u}} \int_{\phi_{\le \eps^\eta}} \frac{\psi(\bx)-\psi(\bx_*)}{\phi(\bx)+\eps} d\bx+\frac{1}{Z_{\bar u}} \int_{\phi_{\le \eps^\eta}^c} \frac{\psi(\bx)-\psi(\bx_*)}{\phi(\bx)+\eps} d\bx\,.
\end{equation}
By assumption~\ref{itm:A3}, when $\eps$ is sufficiently small, $\phi(\bx) \le \eps^\eta$ implies that $b|\bx-\bx_*|^2\le \eps^{\eta/\beta}$, that is, $|\bx-\bx_*|\le b^{-1/2}\eps^{\eta/2\beta}$. This then implies that $|\psi(\bx)-\psi(\bx_*)|\le C |\bx-\bx_*| \le \wt C \eps^{\eta/2\beta}$ for some positive constants $C$ and $\wt C$. Therefore, the first term on the right-hand side can be bounded as
\begin{multline*}
     \left|\frac{1}{Z_{\bar u}} \int_{\phi_{\le \eps^\eta}} \frac{\psi(\bx)-\psi(\bx_*)}{\phi(\bx)+\eps} d\bx\right| \le \frac{1}{Z_{\bar u}} \int_{\phi_{\le \eps^\eta}} \frac{|\psi(\bx)-\psi(\bx_*)|}{\phi(\bx)+\eps} d\bx \\ \le \|\psi(\bx)-\psi(\bx_*)\|_{L^\infty(\phi_{\le \eps^\eta})}\frac{1}{Z_{\bar u}} \int_{\phi_{\le \eps^\eta}} \frac{1}{\phi(\bx)+\eps} d\bx \le \|\psi(\bx)-\psi(\bx_*)\|_{L^\infty(\phi_{\le \eps^\eta})}\le \wt C\eps^{\frac{\eta}{2\beta}}\,.
\end{multline*}

The second term on the right-hand side can be bounded as follows
\begin{multline*}
    \left|\frac{1}{Z_{\bar u}} \int_{\phi_{\le \eps^\eta}^c} \frac{\psi(\bx)-\psi(\bx_*)}{\phi(\bx)+\eps} d\bx\right| \le \frac{2\|\psi\|_{L^\infty(\Omega)}}{Z_{\bar u}} \int_{\phi_{\le \eps^\eta}^c} \frac{1}{\phi(\bx)+\eps} d\bx\\
    \le \frac{2\|\psi\|_{L^\infty(\Omega)}}{Z_{\bar u}} \int_{\phi_{\le \eps^\eta}^c} \frac{1}{\eps^\eta} d\bx\le \frac{2\|\psi\|_{L^\infty(\Omega)}}{Z_{\bar u}} \frac{V_\Omega}{\eps^\eta} \le \bar C \eps^{\frac{2\beta-d}{2\beta}-\eta}\,,
\end{multline*}
where we have used in lower bound~\eqref{EQ:Lower Bounds} of $Z_{\bar u}$ in the last step. Both terms go to $0$ when $\eps\to 0$. Therefore we have shown that  $\dint_\Omega \psi(\bx) \bar u(\bx) d\bx-\psi(\bx_*)\to 0$ as $\eps\to 0$.   This shows that $\bar u(\bx) \to \delta(\bx-\bx_*)$ weakly. The case of $\beta=d/2$ can be proved in exactly the same way if we replace the set $\phi_{\le \eps^\eta}$ by the set $\phi_{\le 1/ \sqrt{\log(1/\eps)}}$ in the above calculations. Indeed, when $\eps$ is sufficiently small, $\phi(\bx) \le (\log(1/\eps))^{-\frac{1}{2}}$ implies that $|\bx-\bx_*|\le b^{-1/2}(\log(1/\eps))^{-\frac{1}{2\beta}}$ which then implies that $|\psi(\bx)-\psi(\bx_*)|\le \wt C (\log(1/\eps))^{-\frac{1}{2\beta}}$ (with $\wt C$ the same as that in the case of $\beta>d/2$). Therefore, we can bound the two terms in the decomposition~\eqref{EQ:Decom Delta Proof} in this case, respectively as
\[
     \left|\frac{1}{Z_{\bar u}} \int_{\phi_{\le 1/ \sqrt{\log(1/\eps)}}} \frac{\psi(\bx)-\psi(\bx_*)}{\phi(\bx)+\eps} d\bx\right| \le  \|\psi(\bx)-\psi(\bx_*)\|_{L^\infty(\phi_{\le 1/ \sqrt{\log(1/\eps)}}}\le \wt C(\log(1/\eps))^{-\frac{1}{2\beta}}\,.
\]
and
\[
    \left|\frac{1}{Z_{\bar u}} \int_{\phi_{\le 1/ \sqrt{\log(1/\eps)}}^c} \frac{\psi(\bx)-\psi(\bx_*)}{\phi(\bx)+\eps} d\bx\right|
    \le \frac{2\|\Psi\|_{L^\infty(\Omega)}}{Z_{\bar u}} \frac{V_\Omega}{(\log(1/\eps))^{-\frac{1}{2}}} \le \bar C (\log(1/\eps))^{-\frac{1}{2}}\,.
\]
Both terms go to $0$ when $\eps\to 0$. The proof is now complete.
\end{proof}

\subsection{Proofs of Theorem~\ref{thm:main} and Corollary~\ref{coro:convergence in p}}

We now prove Theorem~\ref{thm:main}. We split the proof into a few steps.

\begin{lemma}\label{lem:energy bound}
    Let $u$ and $\bar{u}$ be solutions to~\eqref{eq:eqn_u} and \eqref{eq:eqn_u_bar} respectively, and 
    $s(t) := \|u-\bar{u}\|_{L^2(\mu)}$. 
    Then there exists $C>0$ such that 
    \begin{equation}\label{eq:st_1}
        \frac{ds}{dt} \leq - C\eps^2 s - \sqrt{V_\Omega} \, \dfrac{d\eps}{dt}\,  Z_{\bar u}^{-1}\, \eps^{-\frac{3}{2}}\,.
    \end{equation}
\end{lemma}
\begin{proof}
We define $v = u-\bar{u}$ and thus $s(t) =  \|v\|_{L^2(\mu)}:=\Big(\int_\Omega v^2 D_\eps d\bx \Big)^{1/2}$. It is easy to see that $v$ is $\Omega$-periodic and we have, from~\eqref{eq:eqn_u} and~\eqref{eq:eqn_u_bar}, that $v$ solves
\begin{equation}\label{eq:eqn_v}
 \partial_t v(\bx,t)=  \Delta \Big( D_\eps(\bx,t) v(\bx,t) \Big)-\partial_t \bar{u}(\bx,t)\,.
\end{equation}
Multiplying both sides by $D_\eps(\bx,t) v(\bx,t)$, and integrating over the spatial domain $\Omega$ using the periodic boundary condition then leads to the identity
\begin{eqnarray}
\underbrace{ \frac{1}{2} \partial_t \left(  \int_\Omega D_\eps \, |v|^2 d\bx \right)}_{T_1}  +  \underbrace{ \frac{1}{2}\int_\Omega  \left( -\frac{\partial D_\eps}{\partial t} \right) \, |v|^2 d\bx}_{T_2}  &= & \underbrace{ -\int_\Omega | \nabla \left( D_\eps v\right)|^2 d\bx}_{T_3}-\underbrace{ \int_\Omega D_\eps v \,  \bar{u}_t d\bx }_{T_4}  .\label{eq:1234}
\end{eqnarray}

We first observe that term $T_1$ is simply
\begin{equation}
    T_1  = \frac{1}{2} \partial_t \left(\| v \|_{L^2\left(\mu \right)}^2 \right) = s(t)  \,  \frac{d s}{dt}.
\end{equation}
For the term $T_2$, we observe from~\eqref{EQ:Deps} and~\eqref{EQ:Epst} that $\dfrac{\partial D_\eps}{\partial t}=\dfrac{d\eps(t)}{dt}=-\alpha\, (1+t)^{-\alpha-1}\le 0$. Therefore, $T_2 \geq 0$.

To estimate the term $T_3$, we will apply the weighted Poincar\'e inequality in Theorem~\ref{thm:wpi}. In our case, we consider the weight 
$$
    w(x) = D_\eps(\bx,t)^{-1} = \frac{1}{(f(\bx)-f_{\min})^\beta + \eps(t)},
$$ 
in the $A_2$ class (see~\Cref{DEF:Ap}). We can find an upper bound of its $A_2$ constant $[w]_2$ as
\begin{multline}\label{eq:constant_our_setting}
    [w]_2 = \sup_{Q\subset \R^d} \left({\frac {1}{V_Q }}\int _{Q}w (\bx)\,d\bx\right)\left({\frac {1}{V_Q}}\int _{Q}w (\bx)^{-1}\,d\bx\right) \\ \leq \frac{\max_{\bx\in \overline\Omega} w(\bx) }{\min_{\bx\in \overline \Omega} w(\bx)  }  
    = \frac {(f_{\max}-f_{\min})^\beta + \eps(t)}{\eps(t)} \leq \frac {(f_{\max}-f_{\min})^\beta + \eps(0)}{\eps(t)}  = C_\beta\, \eps^{-1},
\end{multline}
where $f_{\max} = \max_{\bx\in\overline\Omega} f(\bx)$, and the constant 
\begin{equation}\label{eq:max D}
    C_\beta = (f_{\max}-f_{\min})^\beta + \eps(0) =(f_{\max}-f_{\min})^\beta + 1= \max_{\bx\in\Omega,\,t\in [0,\infty) } D_\eps(\bx,t)
\end{equation} 
is independent of time $t$. Based on~\eqref{eq:wpi}, for the hypercube $\Omega$ and a Lipschitz function $v$ satisfying $\int_\Omega v \, D_\eps^{-1}  d\bx = 0$, we have
\[
\int_\Omega |v|^2 D_\eps^{-1}  d\bx  \leq \frac{1}{C \eps} \int_\Omega |\nabla v|^2 D_\eps^{-1}   d\bx,
\]
where the constant $C=\left(C_d^2\, \ell_\Omega^2\, C_\beta \right)^{-1}$ with $\ell_\Omega$ the edge length of the hypercube $\Omega$ and $C_d$ the constant introduced in~\eqref{eq:wpi}. We therefore have,
\begin{multline*}
T_3\leq
  -\int_\Omega  | \nabla \left( D_\eps v \right)|^2 \, \frac{\eps}{D_\eps}  d\bx  = -\eps \int_\Omega  | \nabla \left( D_\eps v \right)|^2\, D_\eps^{-1}\, d\bx 
  \leq  - C \eps^2  \int_\Omega  | D_\eps v|^2\, D_\eps^{-1}\, d\bx = -  C\eps^2 s(t)^2.
\end{multline*}

The last term $T_4$ can be bounded from below as follows. We first rewrite the term using $\bar{u} =  Z_{\bar u} D_\eps^{-1}$ from~\eqref{EQ:Stationary Dist}:
\begin{eqnarray*}
T_4 &=& \int_\Omega D_\eps v\, \partial_t \left( Z_{\bar u}^{-1}  D_\eps ^{-1}  \right) d\bx\\
&=& \int_\Omega  v \, \partial_t  \left( Z_{\bar u}^{-1} \right)  d\bx - \int_\Omega v Z_{\bar u}^{-1} \, \frac{dD_\eps}{dt} D_\eps^{-1} d\bx\\
&=& \partial_t  \left(Z_{\bar u}^{-1} \right)\,  \int_\Omega  v \,d\bx  - Z_{\bar u}^{-1}\, \frac{d\eps}{dt}\, \int_\Omega D_\eps^{\frac{1}{2}}v\, D_\eps^{-\frac{3}{2}} d\bx \\
&=& - Z_{\bar u}^{-1}\, \frac{d\eps}{dt}\, \int_\Omega D_\eps^{\frac{1}{2}}v\, D_\eps^{-\frac{3}{2}} d\bx\,,
\end{eqnarray*}
where we have used the facts that $\dfrac{dD_\eps}{dt}=\dfrac{d\eps}{dt}$ and $\dint_\Omega v\, d\bx = 0$ $\forall t>0$.

Using the decomposition $v = v^+ - v^ -$, where $v^+ = \max(v,0)$ and $v^{-} = - \min(v,0)$, we have that $|D_\eps^{\frac{1}{2}}v| = D_\eps^{\frac{1}{2}}v^+ +  D_\eps^{\frac{1}{2}}v^- $. Moreover, it is easy to check that
\begin{align*}%
\int_\Omega D_\eps^{\frac{1}{2}}v\, D_\eps^{-\frac{3}{2}} d\bx
 & \geq D_{\max}^{-\frac{3}{2}} \int_\Omega D_\eps^{\frac{1}{2}}v^+\, d\bx - D_{\min} ^{-\frac{3}{2}} \int_\Omega D_\eps^{\frac{1}{2}}v^-\, d\bx \\
& \geq - D_{\min} ^{-\frac{3}{2}} \|D_\eps^{\frac{1}{2}}v \|_{L^1(\Omega)} \geq - 
\frac{\sqrt{V_\Omega}}{\eps^{\frac{3}{2}}} \| v \|_{L^2\left(\mu \right)}\,,
\end{align*}
where $D_{\max}:=\max_{\bx\in\overline\Omega} D_\eps$, $D_{\min}:=\min_{\bx\in\overline\Omega} D_\eps$, and $d \mu = D_\eps\, d\bx$. This leads to
\begin{equation}
    T_4 \geq  Z_{\bar u}^{-1}\, \dfrac{d\eps}{dt} \,  \eps^{-\frac{3}{2}} \, \sqrt{V_\Omega}\, s(t)\,.
\end{equation}

Finally, we combine all four terms in~\eqref{eq:1234} to have $T_1 = T_3 - T_4 - T_2  \leq T_3 - T_4$. The inequality~\eqref{eq:st_1} then follows. 
\end{proof}

We are now ready to prove the main result~\Cref{thm:main}.

\begin{proof}[Proof of Theorem~\ref{thm:main}]
Using the lower bound on $Z_{\bar u}$ given in Lemma~\ref{LMMA:Well}, as well as the assumption that $\eps(t) = (1+t)^{-\alpha}$, we can further relax~\eqref{eq:st_1} to obtain, after a change of variable $1+t\to t$,
\begin{equation}\label{eq:st_2}
    s_t \leq -C\, t^{-2\alpha} s + C_2\, t^{\ell\alpha - 1},\ \ \ t\ge 1, \qquad \ell = \frac{d}{2\beta} - \frac{1}{2},
\end{equation}
where  $C_2$ is a positive constant and $\beta \geq d/2$. 

Next, we find an upper bound for $s(t)$. First, we discuss the case $\alpha \neq 1/2$.  Define $C_\alpha = \frac{C}{1-2\alpha}$, and 
\begin{eqnarray*}
y(t) &:=& \frac{1}{{2 \alpha -1}} \exp\left(-\frac{C t^{1-2 \alpha }}{1-2 \alpha }\right) = \frac{1}{{2 \alpha -1}} \exp\left(- C_\alpha t^{1-2 \alpha }\right) \, ,  \\
h(t) &:= & \Gamma \left(\frac{\ell \alpha }{1-2 \alpha }, - C_\alpha \, t^{1-2\alpha} \right) = \int_{ - C_\alpha \, t^{1-2\alpha}}^\infty \tau^{\frac{\ell \alpha }{1-2 \alpha } -1 } e^{-\tau} d\tau\, , \\
C_5 &:=& C_2\left(\frac{C}{2 \alpha -1}\right)^{\frac{\ell \alpha }{2 \alpha -1}}  = C_2 \left( - C_\alpha \right)^{\frac{\ell \alpha }{2 \alpha -1}}\, ,\\
C_3 &:=& C_2\left(\frac{C}{2 \alpha -1}\right)^{\frac{\ell \alpha }{2 \alpha -1}} \Gamma \left(\frac{\ell \alpha }{1-2 \alpha },\frac{C}{2 \alpha -1}\right) = C_5\, \Gamma \left(\frac{\ell  \alpha }{1-2 \alpha }, - C_\alpha \right) \, ,\\
C_4 &:=& ( 2 \alpha -1)s(1)  \exp \left(\frac{C}{1-2 \alpha }\right) =( 2 \alpha -1)s(1)  \exp \left(C_\alpha\right)\, ,
\end{eqnarray*}
where $\Gamma (s,x) = \int_{x}^\infty t^{s-1} e^{-t} dt $ is the upper incomplete gamma function. Note that $\Gamma (s,0) = \Gamma(s)$.
The solution to the ODE $\bar{s}_t =  -C t^{-2\alpha} \bar{s} + C_2 t^{\ell \alpha - 1} $ with the initial condition $s(1) = \bar{s}(1)$ is
\begin{equation} \label{eq:true ode sol}
      \bar{s}(t) = (C_4  -C_3) y(t)   + C_5\, y(t) \, h(t) \geq s(t) ,
\end{equation}

If $1-2\alpha <  0$, (i.e., $\alpha > 1/2$), we have 
$$ y(t) \xrightarrow{t\rightarrow \infty}\frac{1}{{2 \alpha -1}},\qquad h(t) \xrightarrow{t\rightarrow \infty}  \Gamma \left(\frac{\ell \alpha }{1-2 \alpha }\right),
$$
both converging to constants. From~\eqref{eq:true ode sol},  $\bar{s}(t)$ converges to a constant and we do not have an upper bound decay for $s(t)$ in this analysis framework.

If $1-2\alpha > 0$, (i.e., $0< \alpha <  1/2$), we have  $y(t) \xrightarrow{t\rightarrow \infty} 0$ exponentially. Since $-C_\alpha \, t^{1-2\alpha} < 0$ for $t\geq 1$, $h(t)$ is a complex-valued scalar with both the real and  imaginary parts  going to $-\infty$ as $t\rightarrow +\infty$. It is worth noting that  $e^{-x} \Gamma(s,-x) \approx x^{s-1}$ when $x$ is sufficiently large, so when $t$ is large, we have
\[
y(t) \, h(t) \approx \frac{1}{2\alpha-1} \left( C_\alpha \, t^{1-2\alpha} \right)^{\frac{\ell \alpha }{1-2 \alpha } -1 } 
= C_6\, t^{(\ell+2)\alpha - 1},
\]
while $y(t) = \frac{1}{2\alpha-1} e^{- C_\alpha t^{1-2 \alpha }} < 0 $ based on its definition and $C_6$ is some positive constant. Thus, $s(t) \leq \bar{s}(t) \lesssim  t^{(\ell+2)\alpha - 1}$.
In order for the upper bound to decay to zero, we need $(\ell+2)\alpha - 1  < 0$, i.e., 
\[
0 < \alpha < \min \left(\frac{1}{2},\frac{1}{\ell + 2} \right) =  \min \left(\frac{1}{2},  \frac{2\beta}{d + 3\beta}\right).
\]

When $\alpha = 1/2$, we need to consider the ODE 
$$
\bar{s}_t =  -C\, t^{-1} \bar{s} + C_2\, t^{\ell/2 - 1} \,,
$$ 
where $s(1) = \bar{s}(1)$  as the initial condition. It has an analytical solution. We then have
\[
s(t) \leq \bar{s}(t) =  \frac{2C_2}{2C + \ell}  t^{\ell/2} + \frac{ s(1)(2C+\ell) - 2C_2}{2C + \ell} t^{-C} \lesssim t^{\ell/2}.
\]
If $\ell = \frac{d}{2\beta} -\frac{1}{2} < 0$, i.e., $\beta > d$, we will have an energy decay when $\alpha = 1/2$. 

To sum up, when $\alpha$, $\beta$ and $d$ are chosen to satisfy
\begin{equation}\label{eq:condition}
   \alpha \in \left(0,\frac{1}{2} \right] \cap \left(0,  \frac{2\beta }{d + 3\beta}\right),
\end{equation}
we have
\[
    s(t) \lesssim t^{(\ell+2)\alpha - 1} =  t^{-\gamma},\qquad \gamma =  1- (\ell+2)\alpha.
\]
This completes the proof.
\end{proof}

The energy estimates in~\Cref{thm:main} allows us to refine the result of~\Cref{LMMA:Concentration}. This is the result of~\Cref{coro:convergence in p}. We now prove this corollary.
\begin{proof}[Proof of~\Cref{coro:convergence in p}]
For a given $\delta > 0$, we have, based on Assumption~\ref{itm:A3} and the lower bound estimations for $Z_{\bar u}$ in~\Cref{LMMA:Well}, and after taking into account that $\eps(t)=(1+t)^{-\alpha} \sim t^{-\alpha}$ for large $t$, that
\[
    \int_{\Omega\cap \cB_\delta(\bx_*)^c} \bar{u}\, d\bx   = Z_{\bar u}^{-1} \int_{\Omega\cap \cB_\delta(\bx_*)^c} \frac{1}{D_\eps(\bx,t)} d\bx \leq Z_{\bar u}^{-1} \frac{V_\Omega}{b\delta^2}  \lesssim   (b\delta^2)^{-1} t^{\alpha(\frac{d}{2\beta}-1) }.
\]
On the other hand, we have
\begin{align*}
\int_{\Omega\cap \cB_\delta(\bx_*)^c} \left( u-\bar{u} \right) d\bx 
    \leq  \int_{\Omega\cap \cB_\delta(\bx_*)^c} |v|\,d\bx
    & \leq  \frac{1}{\sqrt{b\delta ^2 }}\int_{\Omega\cap \cB_\delta(\bx_*)^c} D_\eps^{\frac{1}{2}}\, |v|\, d\bx\\
    & \leq  \frac{\sqrt{V_\Omega}}{\sqrt{b\delta ^2 }} \|D_\eps^{\frac{1}{2}} v\|_{L^2(\Omega)} =  \frac{\sqrt{V_\Omega}}{\sqrt{b\delta ^2 }}\, s(t).
\end{align*}
Therefore, based on the upper bound estimate for $s(t)$ in~\Cref{thm:main}, we have
\begin{multline}\label{EQ:Prob Bound}
    \Prob \left(X_t \not \in \cB_\delta(\bx_*) \right) =  \int_{\Omega\cap \cB_\delta(\bx_*)^c} u\, d\bx= \int_{\Omega\cap \cB_\delta(\bx_*)^c} (u-\bar u)\, d\bx+\int_{\Omega\cap \cB_\delta(\bx_*)^c} \bar u\, d\bx\\
    \leq   C_1 (b\delta^2)^{-\frac{1}{2}} t^{(\frac{d}{2\beta} + \frac{3}{2})\alpha - 1} + \bar{C}_1 (b\delta^2)^{-1} t^{\alpha(\frac{d}{2\beta}-1) } \lesssim 
  t^{-\kappa} ,
\end{multline}
where $C_1, \bar{C}_1$ are positive constants, and $\kappa$ is defined in~\eqref{eq:kappa}. Now if we take $\delta=t^{-\nu}$ with $0<\nu<\min(\gamma, (1-\frac{d}{2\beta})\alpha/2)$, then~\eqref{EQ:Prob Bound} simplifies to
\[
    \Prob \left(X_t \not \in \cB_\delta(\bx_*) \right) \le C_1 b^{-\frac{1}{2}} t^{-(\gamma-\nu)} + \bar{C}_1 b^{-1} t^{-[(1-\frac{d}{2\beta})\alpha-2\nu]}
    \lesssim 
  t^{-\kappa'} ,
\]
where $\kappa'$ defined in~\eqref{eq:kappa2}.
\end{proof}

\subsection{Numerical Experiments}
\label{sec:numerics}

Next, we show a few numerical examples of global optimization to demonstrate the effectiveness of our proposed derivative-free algorithm. 

We will consider minimizing the following objective function $f(\bx)$ on the domain $\Omega = [0,4]^d$ where
\begin{equation}\label{eq:J}
    f(\bx)  = \frac{1}{d\, \bar{f}} \left( 0.3|\bx-2|^2 - \sum_{i=1}^d  \cos (4x_i -8) + d \right),\quad \bx = [x_1,\ldots,x_d]^\top,
\end{equation}
where $\bar{f} = 2.3455$ so that $f_{\max} =  \max_{\bx\in[0,4]^d} f(\bx) = 1$ for any dimension $d$. There is a unique global minimum of $f(\bx)$ at $\bx_* = [2,\ldots, 2]^\top \in \Omega$ with function value $f_{\min}=f(\bx_*)=0$. The shapes of the objective function in dimension $d=1$ and $d=2$ are illustrated in~\Cref{fig:obj_J}.
\begin{figure}[!htb]
    \centering
    \includegraphics[height = 4.5cm]{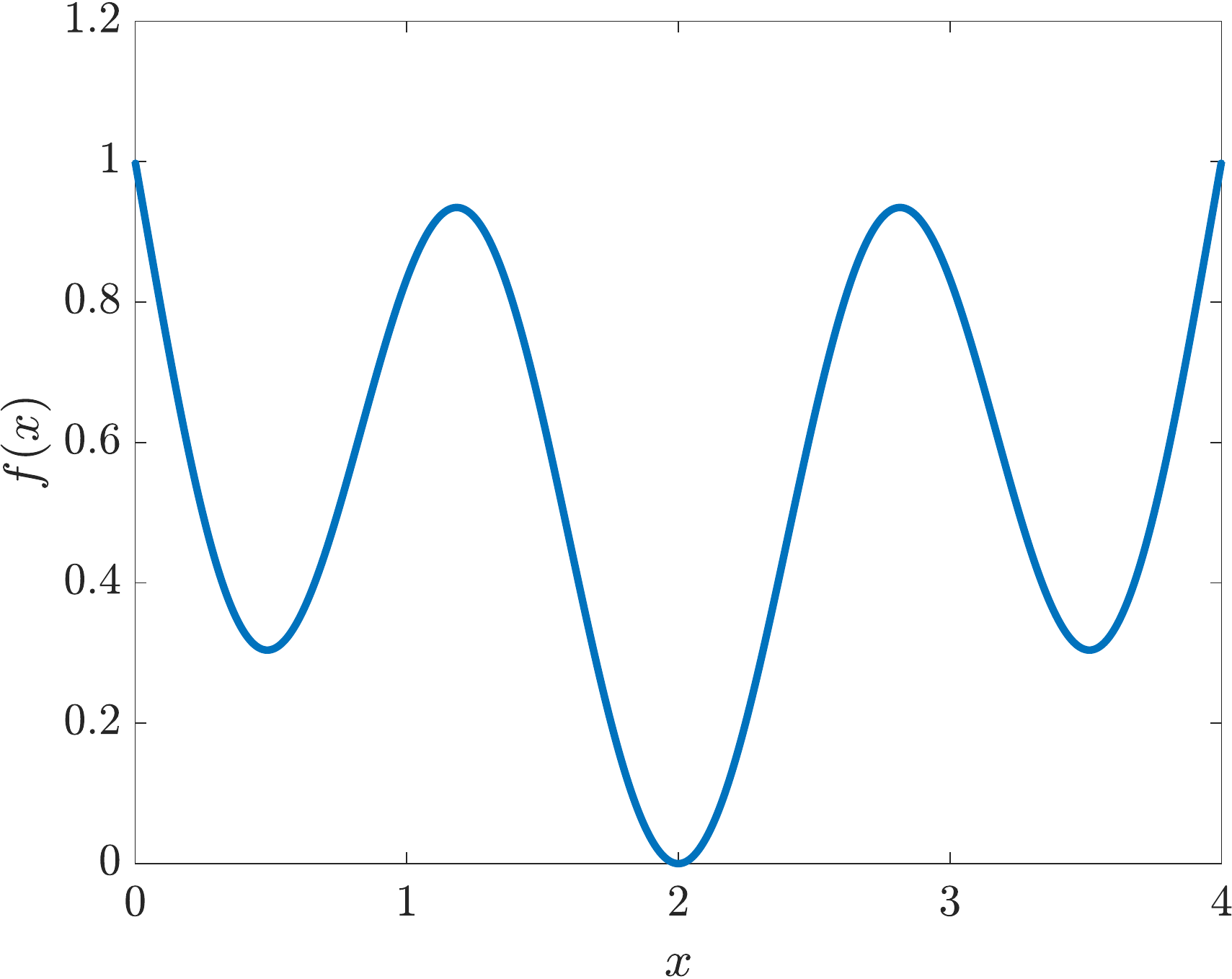} \hspace{1cm}
    \includegraphics[height = 4.5cm]{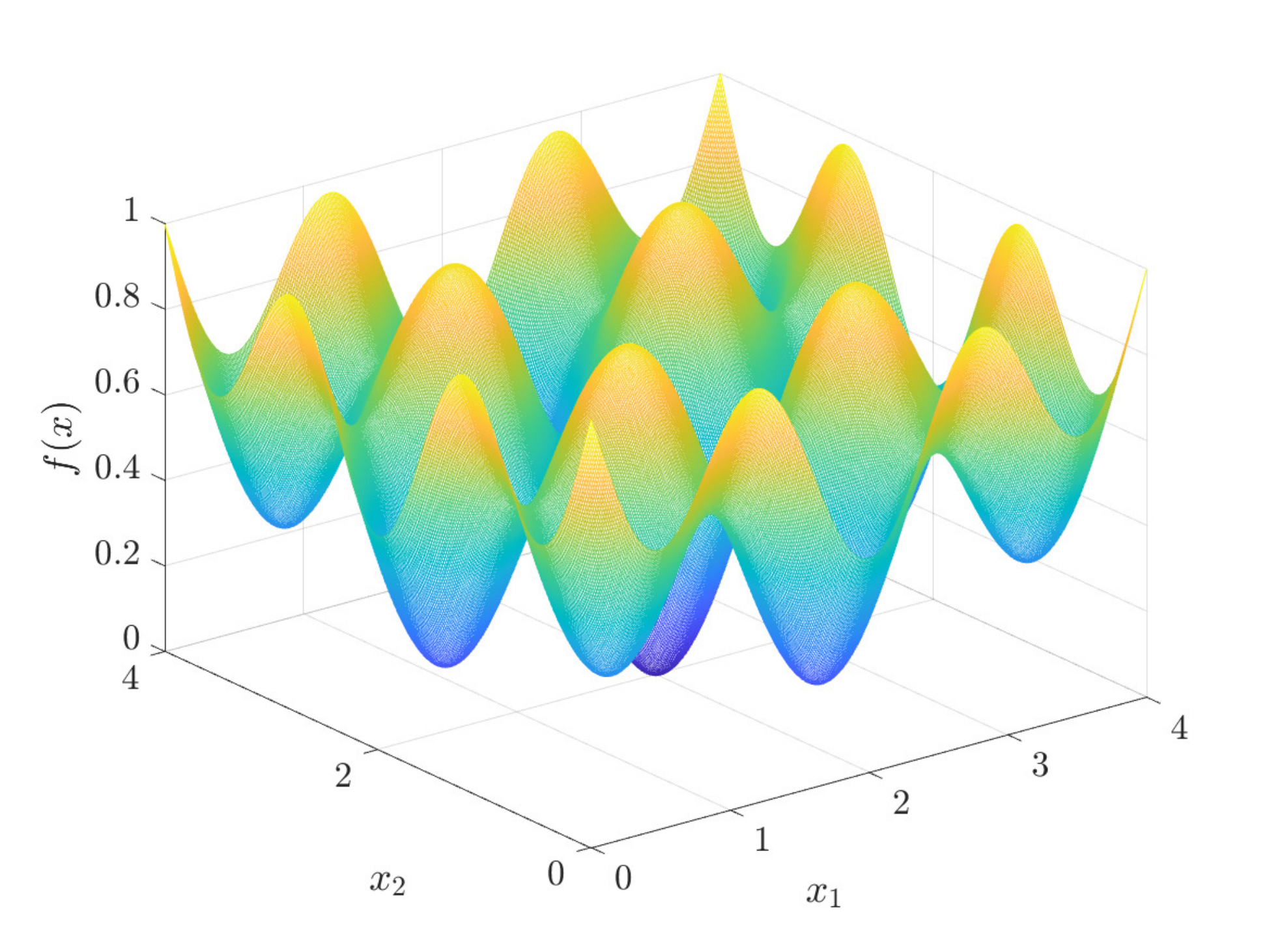}
    \caption{Optimization landscapes of the  objective function $f(\bx)$ in~\eqref{eq:J} in dimension $d=1$ (left) and $d=2$ (right).}
    \label{fig:obj_J}
\end{figure}

Our numerical simulations are based on the discrete algorithm~\eqref{EQ:Adaptive A}, where we fix the step size $\eta$ to be a constant. The standard deviation for the noise is taken as the discrete equivalence of~\eqref{EQ:Sigma}, i.e.,
\begin{equation}\label{EQ:Disc Sigma}
\sigma_n =\sqrt{2\Big[\Big((f(X_n)-f_{\min})^+\Big)^\beta+\eps\Big]}\,,\quad \eps = c n^{-\alpha}\,,
\end{equation}
where $c = 10^{-3}$ is a fixed scalar. This setup corresponds to the main results of the paper proved in~\Cref{SEC:Theory}. The update rule for the iterate is
\begin{equation}\label{eq:discrete_update}
    X_{n+1} = X_n + \sigma_n \xi_n,
\end{equation}
where $\xi_n\sim \mathcal{N}(0,I_d)$, the standard normal distribution on $\mathbb{R}^d$ with an enforced periodic boundary condition.
The standard deviation $\sigma_n$, or equivalently, the diffusion coefficient, is both state- and time-dependent.

\begin{figure}[!htb]
    \centering
    \includegraphics[width = 1.0\textwidth]{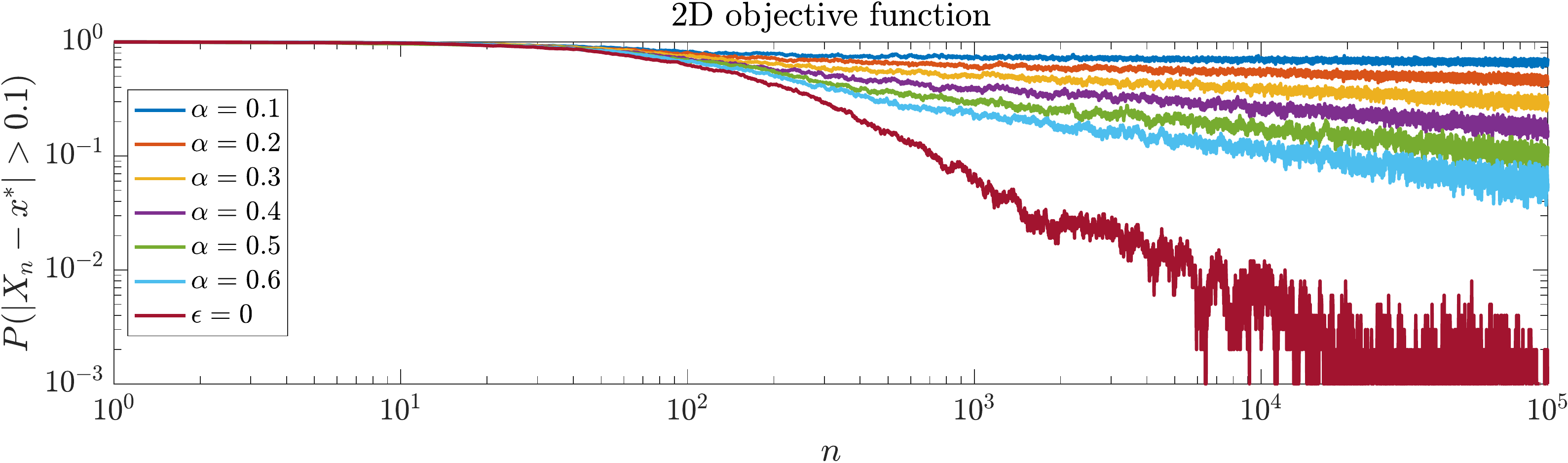}
    \caption{Convergence history of iteration~\eqref{EQ:Adaptive A} with $\sigma$ given in~\eqref{EQ:Disc Sigma} in the case of minimizing $f(\bx)$ of~\eqref{eq:J} in dimension $d=2$ with $\beta = 2$. Shown are results for different values of $\alpha$ after $10^5$ iterations.}
    \label{fig:2D alpha}
\end{figure}
The convergence histories are shown in~\Cref{fig:2D alpha}, for the case $d=2$ with $\beta=2$ where we vary the value of the parameter $\alpha$. Based on the log-log plots, we see that the choice of $\alpha$ directly affects the convergence speed in the discrete algorithm as the bigger the $\alpha$, the faster the convergence. 
While we will discuss more on the case of $\eps=0$ in~\Cref{SUBSEC:eps=0}, for the purpose of comparison, we include in~\Cref{fig:2D alpha} a plot for the case where the term $\eps = 0$ as the limit of $\alpha\rightarrow \infty$.

\section{Practical Generalizations}
\label{SEC:Gen}

The numerical results we presented in the previous section verified our theoretical analysis in \Cref{SEC:Theory}, where we assumed that the value of the global minimum of the objective function, $f_{\min}$,  is known and demonstrated that the algorithm could perform well in more complex situations. In this section, we provide further discussions on practical situations under which our algorithm performs almost as well as in the ideal case. 

We start with a numerical illustration for various cases regarding $f_{\min}$ and the value of $\eps$, which will be further discussed in~\Cref{SUBSEC:Unknown fmin,SUBSEC:eps=0} respectively. The left plots of \Cref{fig:2D epsilon compare} show single-run trajectories of the cases $\eps=0$ (top) and $\eps\neq 0$ (bottom) respectively, under the setting that $f_{\min}$ is known (and $f(\bx_*) = 0$). The case of $\eps=0$ is superior in stabilizing the iterates around the global minimum.  The right plots of~\Cref{fig:2D epsilon compare} are results on the trajectory of one single run with $\eps= 0$ when $f_{\min}$ is unknown (and estimated with the method in~\Cref{SUBSEC:Unknown fmin}). The iterates get stuck at a wrong position very quickly, showing that we cannot set $\eps=0$ when $f_{\min}$ is unknown, in contrast to the two left plots. We will elaborate on these observations more in this section.

\begin{figure}[!htb]
\centering
\subfloat[history of the iterates from a single run]{
\includegraphics[width = 0.43\textwidth]{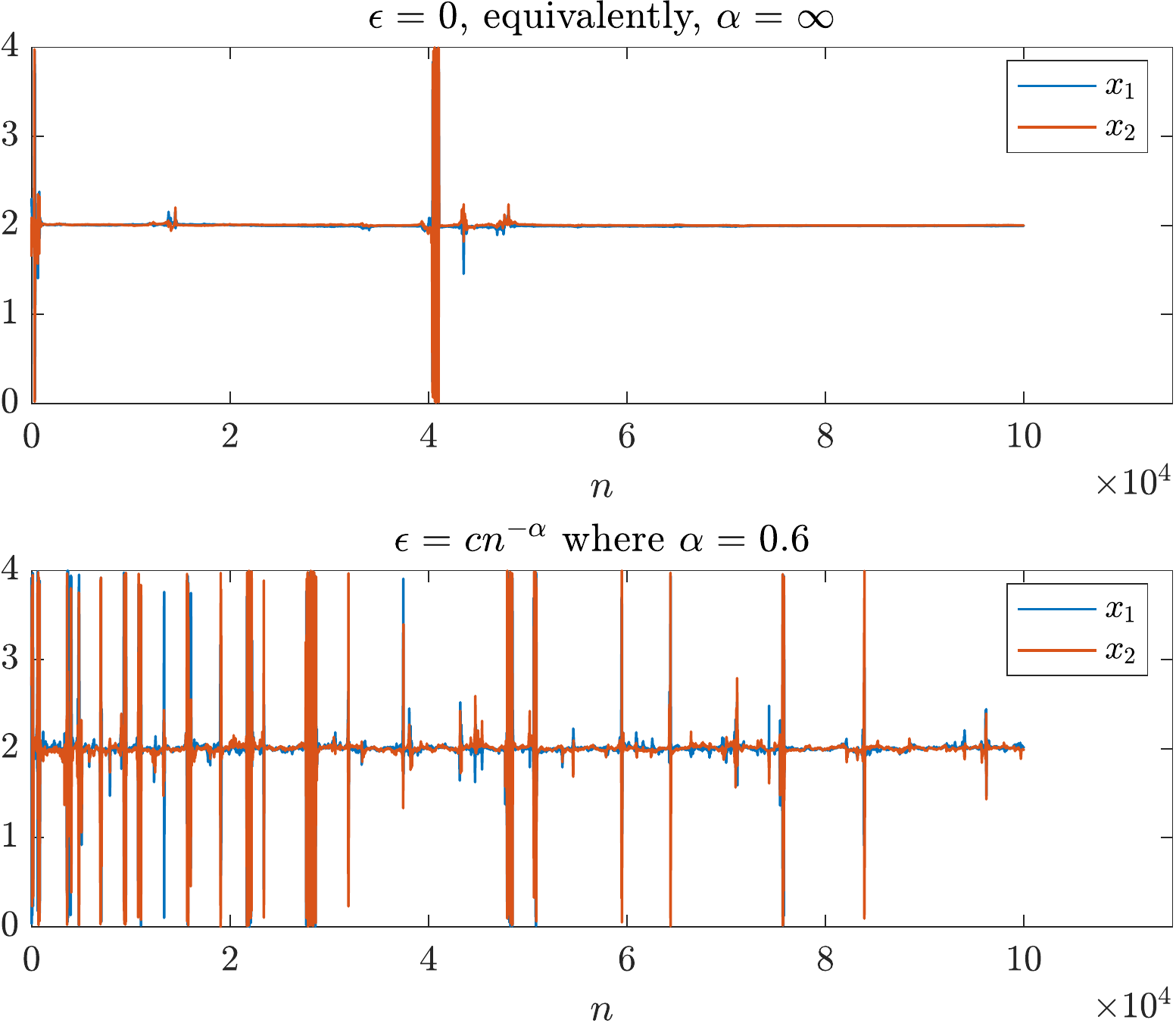}\label{fig:2D epsilon compare 1}}
\subfloat[trajectory of $X_n$, $f_{\min}^n$ and $D_n(X_n) = (\sigma_n)^2/2$]{\includegraphics[width  = 0.49\textwidth]{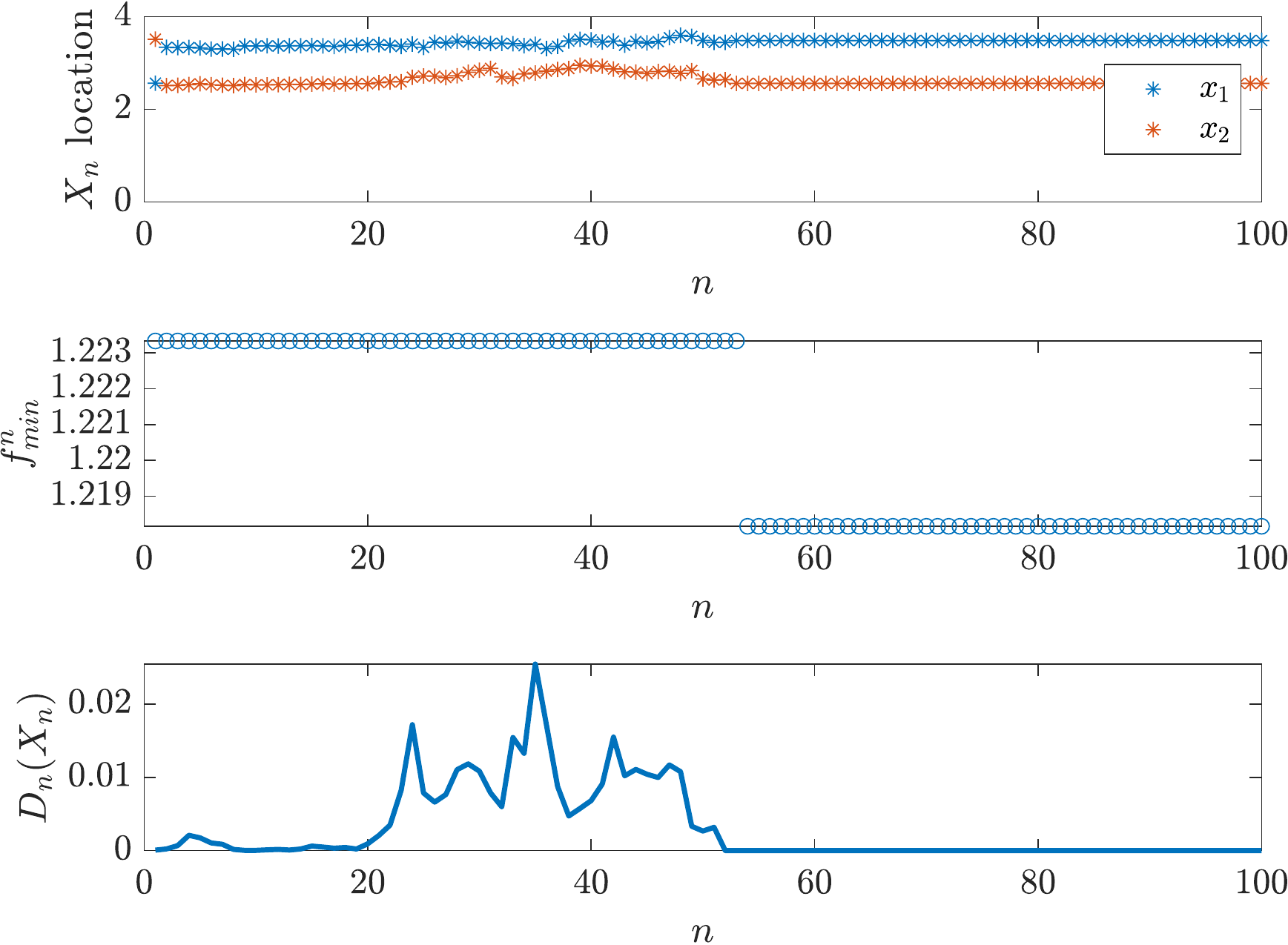}\label{fig:2D epsilon compare 2}}
\caption{(a): the history of the iterates $\{X_n\}$ from a single run when $\eps=0$ (top) and $\eps = cn^{-\alpha}$ (bottom) respectively, with $f_{\min}$ given. (b): the trajectory of the iterates $X_n$, the estimated minimum value $f_{\min}^n$ and the effective diffusion coefficient $D_n(X_n) = (\sigma_n)^2/2$ from one single run with $\eps= 0$ and $f_{\min}$ unknown. The global minimum is $[2,2]^\top$ in both cases.}
    \label{fig:2D epsilon compare}
\end{figure}


\subsection{Estimating optimal objective function value}
\label{SUBSEC:Unknown fmin}

Previously, and also in the main theoretical results, we have used the assumption that the value $f_{\min}:=f(\bx_*)$ (but not the location $\bx_*$) is known \emph{a priori}. This is often true in many applications (for instance, data matching) where $f_{\min}=0$. When $f_{\min}$ is unknown, developing a convergence theory for the algorithm is much more challenging. The difficulty is that the above analysis is in the continuum and the estimation on $f(\bx_*)$ is inherently discrete. However, with a little more effort in estimating $f_{\min}$ during the iteration, we can make our algorithm efficient under such a situation.

It is important to have the state-dependent term in the algorithm, which is the only component that encodes any information regarding the objective function $f(\bx)$. We may consider a different variant of~\eqref{EQ:Sigma} and~\eqref{EQ:Disc Sigma}:
\begin{equation}\label{eq:no-Jmin-diff}
    \sigma_n=\sqrt{2\Big[\Big(f(X_n)-f_{\min}^n\Big)^\beta+ \eps \Big]},\ \ \ \mbox{where}\ \ \ f_{\min}^n:=\min \left\{f(X_n), f_{\min}^{n-1} \right\}\,,\, \eps  = c n^{-\alpha}\,.
\end{equation}
The role of $f_{\min}^n$ here is to approximate $f(\bx_*)$ through the history minimum of the objective function values from the past iterates. We also need to have $\eps \neq 0$ not only to avoid $\{X_n\}$ stagnating at any history minimum rather than the global minimum but also to visit everywhere of the domain $\Omega$; see for a counterexample in~\Cref{fig:2D epsilon compare 2}.

\begin{figure}[!htb]
    \centering
    \includegraphics[width = 1.0\textwidth]{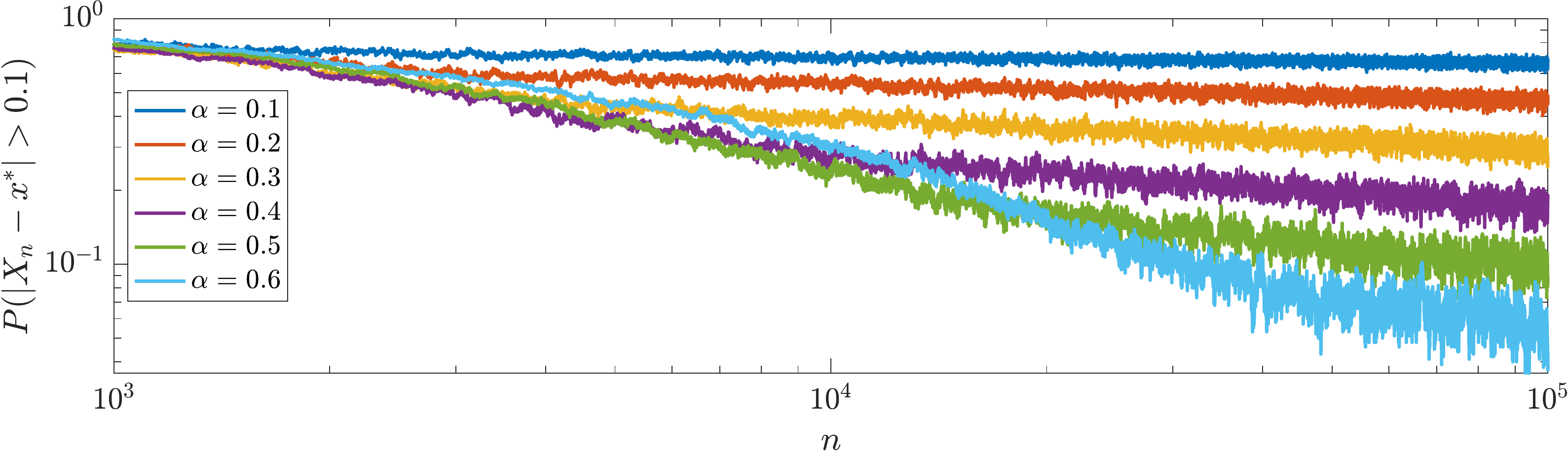}
    \caption{The convergence performance of~\eqref{EQ:Adaptive A} with the diffusion coefficient defined in~\eqref{eq:no-Jmin-diff} and $\eps = 10^{-3} n^{-\alpha}$. The minimum objective function value $f(\bx_*)$ is estimated by $f_{\min}^n$.}
    \label{fig:2D estimate Jmin}
\end{figure}

On the other hand, in the case of $f(\bx_*)$ unknown, when we set $\eps\neq 0$ but monotonically decaying as $n$ becomes large, we observe in~\Cref{fig:2D estimate Jmin} the decay of $\mathbb{P}(|X_n - \bx_*| > 0.1)$ as $n$ increases, but much slower than the cases shown in~\Cref{fig:2D alpha} in which we assume to know $f(\bx_*) = 0$ \emph{a priori}. In~\Cref{fig:2D Jmin comp 1}, we present the history of $f_{\min}^n$ and $X_n$ from a single run when $\eps = 10^{-3} n^{-0.6}$. In~\Cref{fig:2D Jmin comp 2}, we show a comparison regarding whether $f(\bx_*)$ is known \emph{a priori} or not where the convergence performances are estimated from $10^3$ i.i.d.~runs and $\eps = 10^{-3} n^{-0.6}$ in both cases.
\begin{figure}[!htb]
\centering
\subfloat[the history of the iterates $\{X_n\}$ and $\{f_{\min}^n\}$]{\includegraphics[width = 0.46\textwidth]{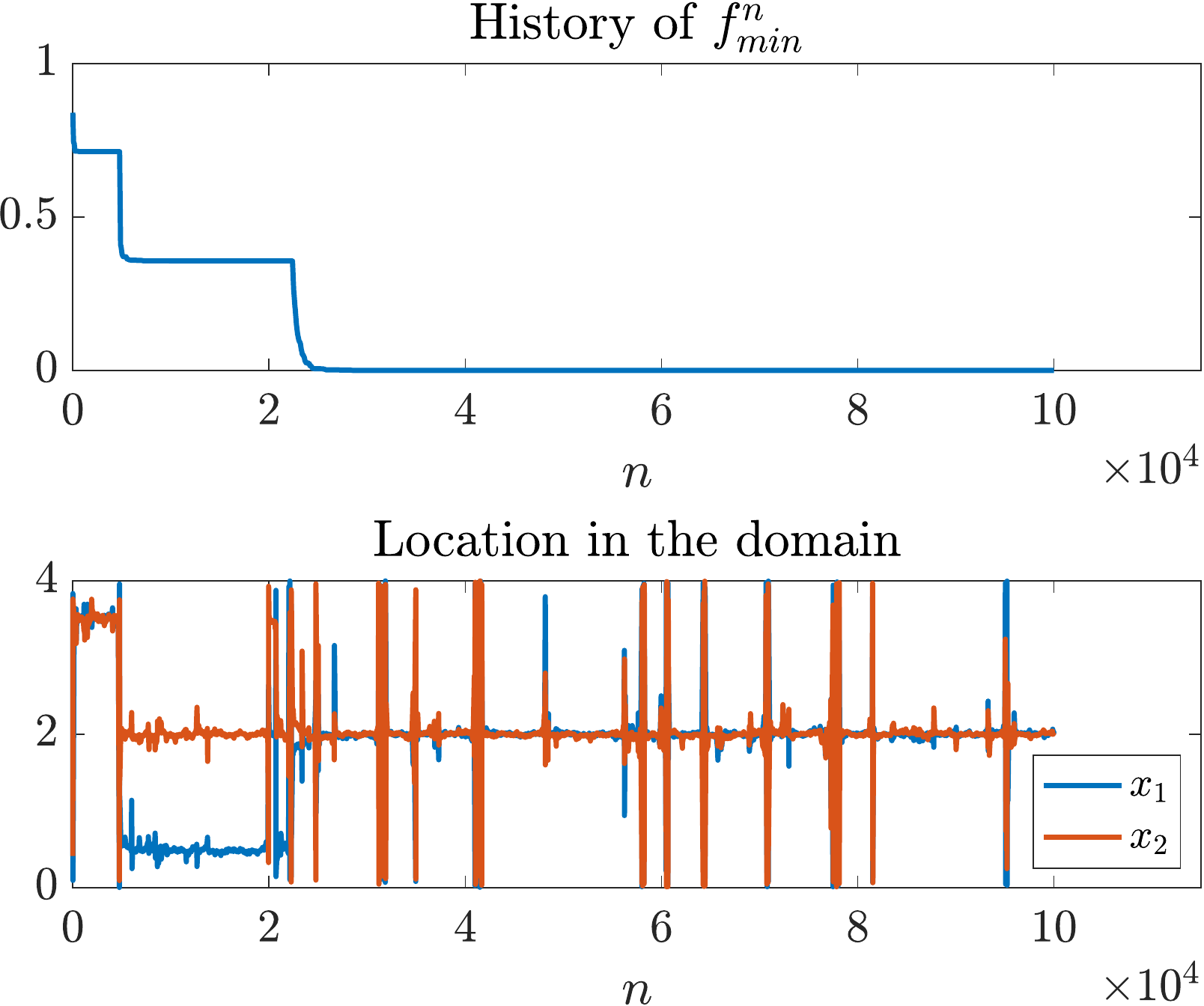}\label{fig:2D Jmin comp 1}}
\subfloat[comparison of convergence regarding $f_{\min}$]{\includegraphics[width = 0.46\textwidth]{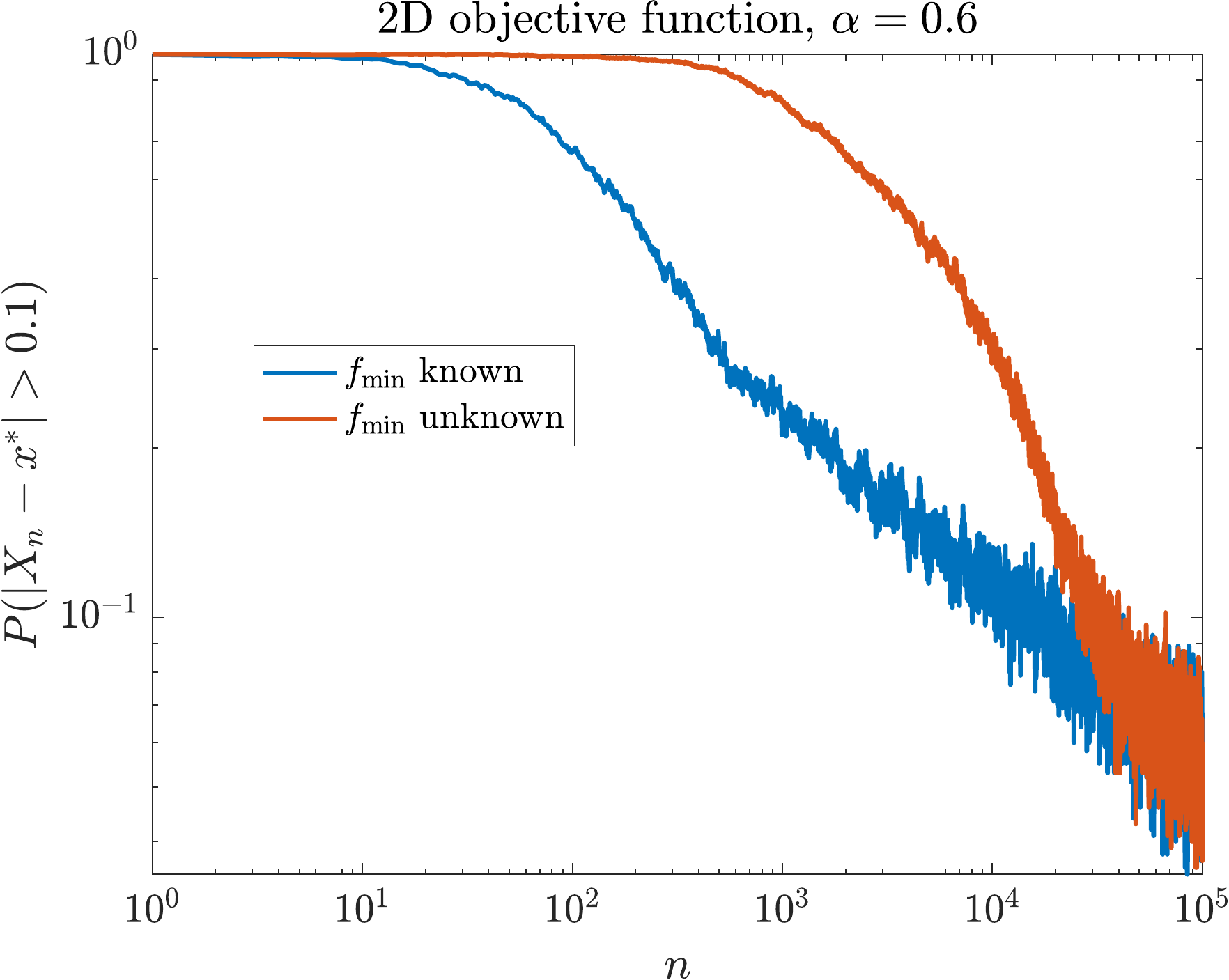}\label{fig:2D Jmin comp 2}}
\caption{(a): the history of the iterates $\{X_n\}$ and $\{f_{\min}^n\}$ from a single run when $\eps = 10^{-3} n^{-0.6}$ in~\eqref{eq:no-Jmin-diff}. (b): the comparison of convergence performances in terms of whether $f_{\min} = f(\bx_*)$ is known or needs to be estimated using $f_{\min}^n$ following~\eqref{eq:no-Jmin-diff}. In both cases, we set $\eps = 10^{-3} n^{-0.6}$.}
\label{fig:2D Jmin comp}
\end{figure}

The proposed algorithm~\eqref{eq:no-Jmin-diff} for estimating the optimal value of the objective function $f_{\min}$ is based on discrete $f(X_n)$ values and does not fit well into the continuum style convergence proof. With increasing values of $n$, it is possible to approximate $f_{\min} = f(\bx_*)$ with increasing accuracy. In the numerical example related to~\Cref{fig:2D Jmin comp 1}, the simple estimate~\eqref{eq:no-Jmin-diff}  of $f(\bx_*)$ was used. The figure shows the convergence to $\bx_*$ and the optimum estimate to $f(\bx_*)$.
There are abnormal cases where the estimate~\eqref{eq:no-Jmin-diff}  would require very slow decay of $\eps(t)$,
and for a rigorous convergence result, we adopt the same strategy, which we used in~\cite{EnReYa-arXiv22}, of
basing the hyperparameter estimates on extra sampling. If in the sequence in~\eqref{eq:no-Jmin-diff}  we add
uniformly sampled values $\{Y_n\}$ from the domain $\Omega$, we can guarantee almost-sure convergence of  $f(Y_n)$ to the optimal value $f({\bf x_*})$; see~\Cref{prop:visit_ae} below.

\begin{proposition}\label{prop:visit_ae}
Assume that there is a subset $\Omega_{sc} \subseteq \Omega$ on which the objective function $f(x)$ is strongly convex and ${\bx_*} \in \Omega_{sc}$. Define the monotone-decreasing sequence 
$$f_{\min}^n:=\min \left\{f(X_n), f(Y_n), f_{\min}^{n-1} \right\},\quad f_{\min}^0 = \min \left\{f(X_0), f(Y_0)\right\},$$
where $\{X_n\}$ are iterates from~\eqref{eq:discrete_update} with $\sigma_n = \sqrt{2(f(X_n)-f_{\min}^n)^\beta+ 2\eps }$ and $\{Y_n\}$ are uniform samples drawn from the domain $\Omega$. Then we have $f_{\min}^n\xrightarrow{n\rightarrow \infty} f({\bf x_*})$ almost surely.
\end{proposition}
\begin{proof}
Let $\delta$ be the largest positive constant such that $\Omega_\delta := \{x : f(x) - f({\bx_*})  \leq \delta \} \subseteq \Omega_{sc}$. Note that $\Omega_{\delta}$ is nested between two ellipsoids centered at $\bx_*$ and the ratio $|\Omega_{\delta}|/|\Omega| \leq C \delta^{d/2}$ for some positive constant $C$~\cite[Eqn.~(3.22)]{EnReYa-arXiv22}.  Also, it is easy to see that
\[
f_{\min}^{n} \leq \min \left\{f(Y_1),f(Y_2), \ldots, f(Y_n) \right\} := M^{(n)}.
\]
Therefore, we have
\begin{eqnarray}
\Prob\left( \lim\limits_{n\rightarrow \infty} f_{\min}^{n}  - f({\bf x_*}) > \delta  \right)   \leq \Prob\left( \lim\limits_{n\rightarrow \infty}M^{(n)}  - f({\bf x_*}) > \delta  \right) \nonumber 
&=&  \Prob\left( \cap_{n=0}^\infty \{ Y_n \not\in \Omega_\delta  \} \right)   \nonumber \\
&=& \prod_{n=0}^\infty \Prob (Y_n\not \in \Omega_{\delta}) \leq  \lim_{n\rightarrow \infty} \left(1 - C\delta^{d/2}   \right)^n = 0.  \label{eq:as_1}
\end{eqnarray}
Since~\eqref{eq:as_1} holds for any $0 < \delta' \leq \delta$, we conclude that $f_{\min}^{n} $ converges to $f({\bf x_*})$ almost surely.
\end{proof}

\subsection{Regularization-free algorithm}
\label{SUBSEC:eps=0}

In this section, we discuss the case when $\eps = 0$ in~\eqref{EQ:Sigma}. 
Based on the definition in~\eqref{EQ:Sigma}, $\sigma(f)$ is not integrable for $\beta \geq d/2$. This means that, mathematically, the process $\{X_t\}_{t\ge 0}$ can get arbitrarily close to the global minimizer but will never reach it unless $X_0 = {\bf x_*}$. However, from a practical point of view, being arbitrarily close is sufficient. 

Numerically, we still observe a rapid convergence of the discrete algorithm~\eqref{EQ:Adaptive A} to the global minimizer in this case, even though this ``convergence'' might not be in a strict mathematical sense since we have finite spatial resolution when computing the distributions. To be more precise, we set
\[
    \sigma_n(X_n) = \sqrt{2\Big[(f(X_n)-f_{\min}^*)^+\Big]^\beta}\,,
\]
which is~\eqref{EQ:Sigma}. In~\Cref{fig:2D eps=0}, we plot the convergence histories for $d=2$ and $\beta$ ranging from $0.5$ to $2$, and we assume $f_{\min}^* = f_{\min} = 0$ is known. The probabilities in the $y$-axis are estimated using $10^3$ i.i.d.~runs while the $x$ axis is the number of iterations. The initial guess is uniformly sampled from the domain $\Omega$. It is worth noting that when $\beta < 1 = \frac{d}{2}$, there is no guarantee for convergence in probability since $\dlim_{\eps \rightarrow 0} Z_{\bar u} <\infty$ as defined in~\eqref{EQ:Stationary Dist}.

Next, we consider another optimization problem with four different global minima, whose optimization landscape is seen in~\Cref{fig:2D eps=0 dw 1}. We denote the global minimizers by $x_1^* = [2,2]^\top$,  $x_2^* = [-2,-2]^\top$,  $x_3^* = [2,-2]^\top$ and  $x_4^* = [-2,2]^\top$. We implement the same algorithm with $\eps =0$ and $\beta = 4$, assuming again $f_{\min} = 0$ is known. The convergence behavior is shown in~\Cref{fig:2D eps=0 dw 2}. We can see that there are equal probabilities of roughly $25\%$ for the iterate $X_n$ to be in a close neighborhood of any of the four global minima for $n$ large enough.

\begin{figure}[!htb]
    \centering
    \includegraphics[width = 0.9\textwidth]{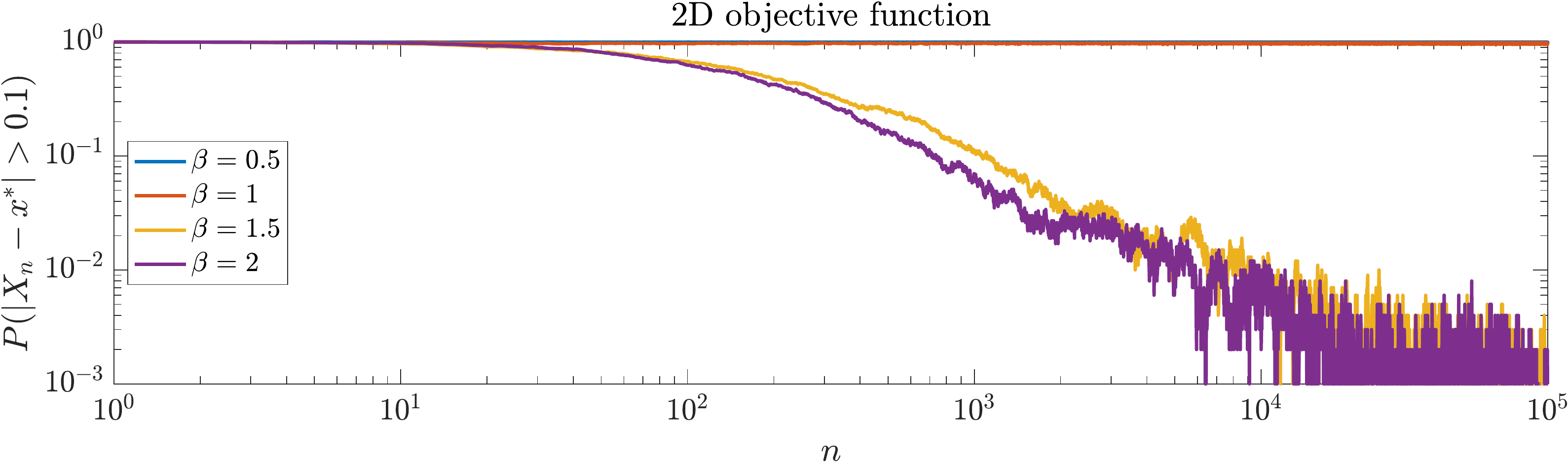}
    \caption{Convergence history for minimizing~\eqref{eq:J} with $\eps = 0$ and $d=2$ after $10^5$ number of iterations.}
    \label{fig:2D eps=0}
\end{figure}

\begin{figure}[!htb]
    \centering
    \subfloat[optimization landscape with 
    $4$ global minimizers]{\includegraphics[width = 0.49\textwidth]{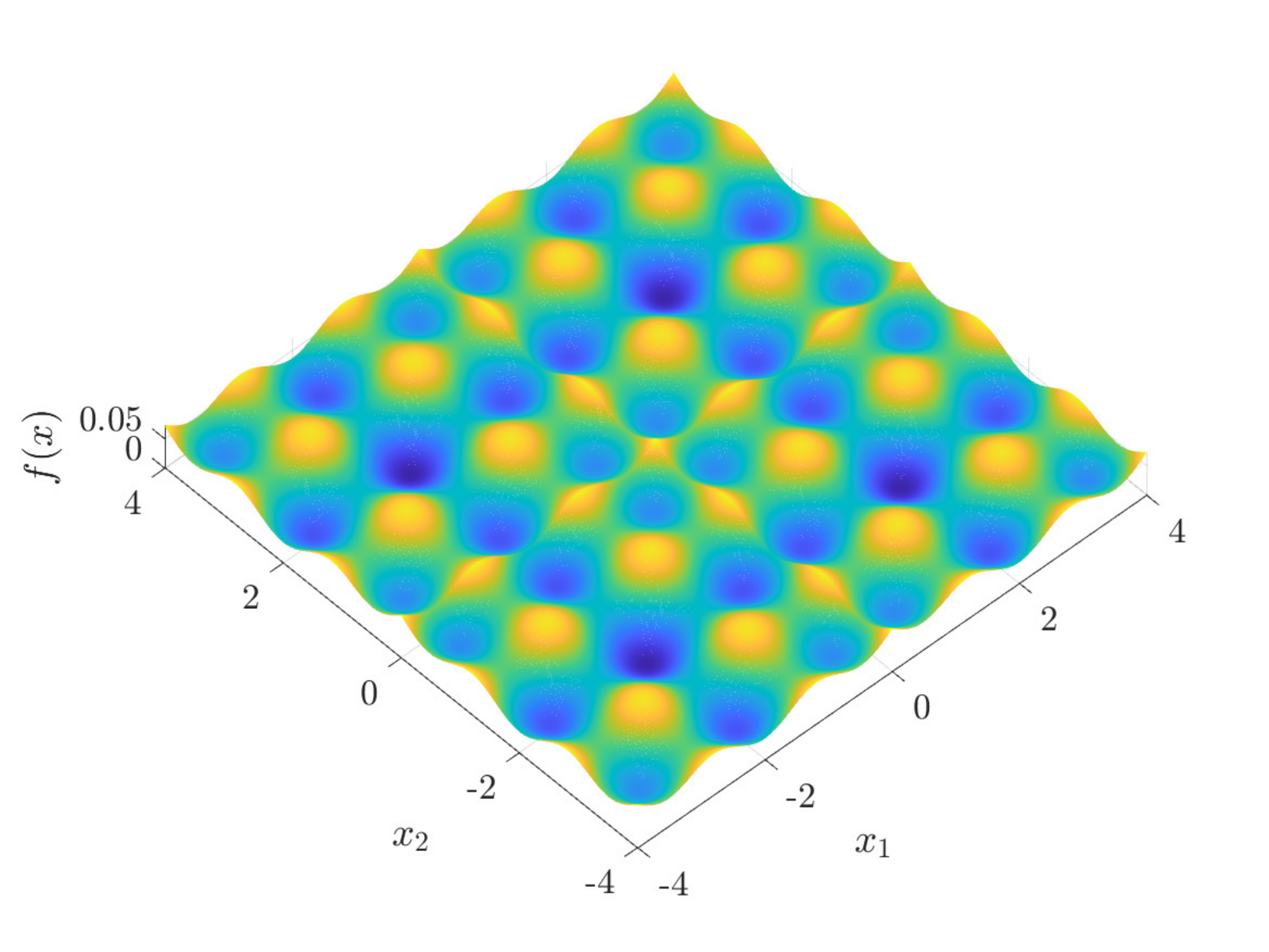}\label{fig:2D eps=0 dw 1}}
\subfloat[convergence performance with $\eps = 0$]{\includegraphics[width = 0.49\textwidth]{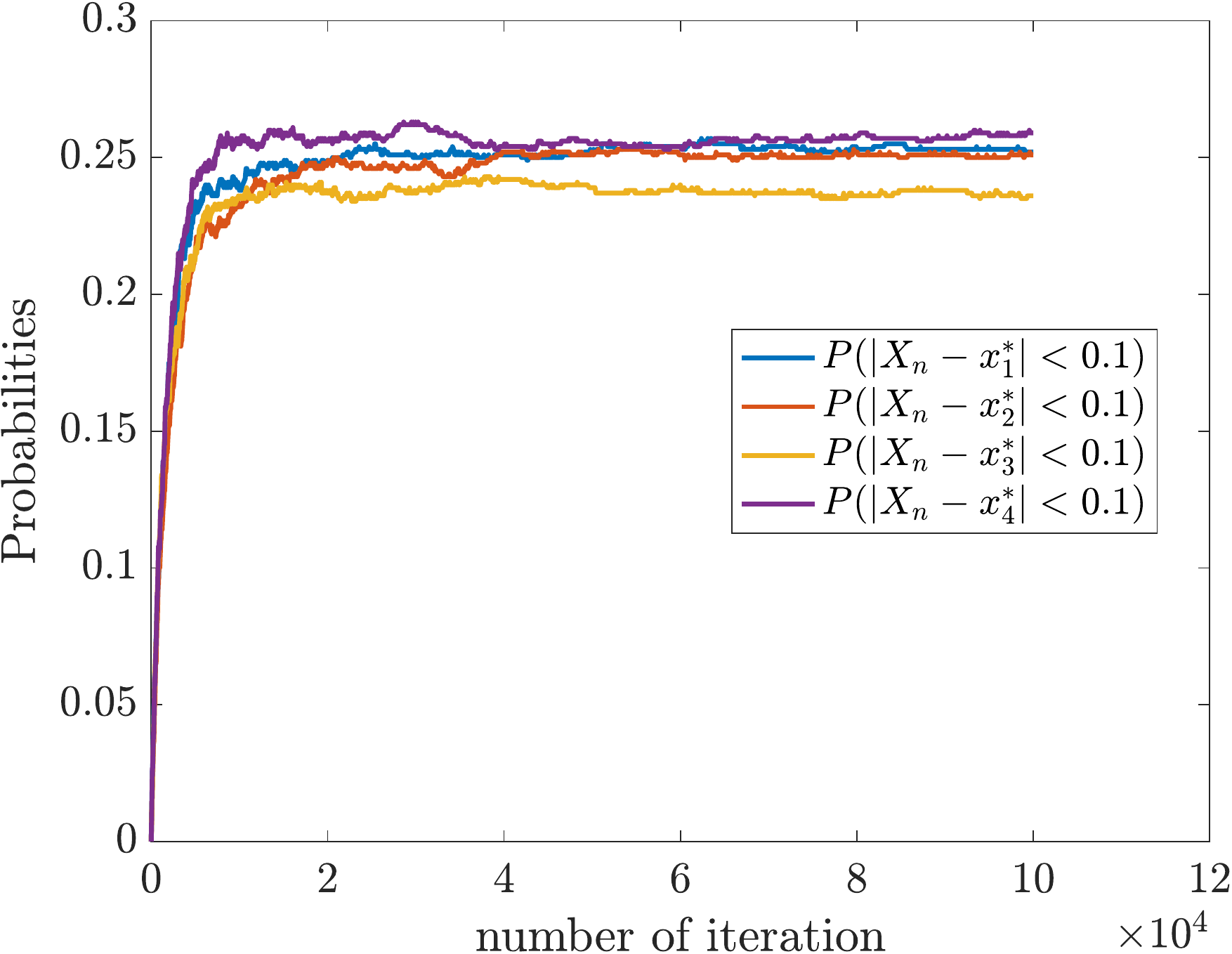}\label{fig:2D eps=0 dw 2}}
\caption{(a): an objective function landscape with four global minimizers labeled as $x_1^* = [2,2]^\top$,  $x_2^* = [-2,-2]^\top$,  $x_3^* = [2,-2]^\top$ and  $x_4^* = [-2,2]^\top$; (b): convergence performance (from $10^3$ i.i.d.~runs) of the proposed algorithm with $\eps = 0$, $\beta = 4$, and $d=2$ after $10^5$ number of iterations.}
\label{fig:2D eps=0 dw}
\end{figure}

The numerical experiment in~\cref{fig:2D alpha} shows that eliminating the regularization term $\eps > 0$ in the algorithm~\eqref{EQ:Disc Sigma} gives a faster convergence rate than that with the regularization term, at least for that particular objective function $f$. This requires that $f_{\min} = f(\bx_*)$ is known. With an unknown optimal objective function value, there is a clear risk of having the algorithm trapped in local minima; see~\Cref{fig:2D epsilon compare 2}. The regularization-free method, i.e., $\eps=0$, works very well when the optimization landscape is convex or when $f(\bx_*)$ is known.

{\bf Lack of theoretical understanding for the case of $\eps=0$.} 
We currently have a minimal theoretical understanding of the $\eps=0$ algorithm due to the strong degeneracy of the diffusion coefficient $D$ in this case. The proof from~\Cref{SEC:Theory} does not apply here because of the lack of appropriate Poincar\'e inequality in the strongly degenerate case, i.e., $\beta \geq d/2$. What we observe in the simulations might be an effect of discretization in the computational algorithm. 

The degenerate elliptic operator in~\eqref{eq:generator} is a challenge discussed extensively in the PDE and SDE literature; see, for example~\cite{ChSt-SIAM10,EpMa-SIAM10,FrGoGoRo-EJDE12,FrMu-Book2016,heinonen2018nonlinear} and references therein. The classical way of handling degeneracy is to regularize the problem with a parameter $\eps$ and then take $\eps\to 0$~\cite{Oleinik-MS66,StVa-Book97}. This allows one to establish the existence, and sometimes uniqueness, of the solution in a finite time interval $(0, T]$ but does not generalize to the limiting case of $T\to \infty$. There are recent results based on weighted estimates for the problem in the absence of the regularization parameter $\eps$, mainly for the case of weak degeneracy, that is, when the exponent $\beta$ is sufficiently small (see, for instance, reference~\cite{FrGoGoRo-EJDE12} for a more precise definition of weak and strong degeneracy)~\cite{FrGoGoRo-EJDE12,FrMu-Book2016,heinonen2018nonlinear}. In most cases, the existence of solutions to the Fokker--Planck equation~\eqref{EQ:Fokker--Planck} can only be established in the one-dimensional case (again in specific weighted function spaces) for a finite time interval $(0, T)$. 

{\bf Existing results in simplified settings.} There are indeed some precise characterizations of the singular behavior of such degenerate problems in simplified (yet still difficult) scenarios where the particular forms of diffusion coefficients (such as $D=x(1-x)$ on $(0, 1)$) are assumed, for instance, in the case where the point of degeneracy (that is, the global minimizer in our case) is on the boundary of the domain and appropriate boundary conditions are prescribed at the point of degeneracy; see for instance~\cite{ChSt-SIAM10,EpMa-SIAM10} for the detailed analysis of the Wright--Fisher equation. To demonstrate how the specific structure of the problem plays a role in the theory, let us consider the one-dimensional case of $f(x)=x^2$ and $\beta=1$. We further simplify the problem by taking $\Omega=\bbR$. With all these simplifications, we have that $D=x^2$, and the Fokker--Planck equation~\eqref{EQ:Fokker--Planck} simplifies to
\begin{equation}
    u_t =(x^2 u)_{xx},\ \ -\infty<x<+\infty\,.
\end{equation}
If we introduce the new variable $v=x^2 u$, we can check that $v$ solves
\begin{equation*}
    v_t =x^2 v_{xx},\ \ -\infty<x<+\infty\,.
\end{equation*}
Due to the degeneracy at $x=0$, we have that $v(0, t)=0$. Therefore, we can focus only on the positive axis. The equation for $v$ can be written as
\begin{equation}
    v_t =x^2 v_{xx},\ \ 0<x<+\infty, \ \ \ v(0, t)=0\,.
\end{equation}
Let us perform the change of variable $x=e^y$, that is, $y=\log x$. Then it is easy to check that the interval $(0,+\infty)$ is mapped to $(-\infty, +\infty)$. The Fokker--Planck equation is now mapped into the following constant-coefficient form:
\begin{equation}
    \wt v_t=\wt v_{yy} - \wt v_y,\ \ \ \wt v(-\infty, t)=0\,.
\end{equation}
With the boundary conditions, this system has a non-localized stationary distribution. This leads to the fact that $\int_{y_L}^{y_R} \wt v dy\nrightarrow 0$ as $t\to\infty$ for any finite interval $(y_L, y_R)$. Using the fact that $v=x^2 u$, we conclude that $\int_{x_L}^{x_R} u dx \le c\,x_R^{-2} \to 0$ as $x_R\to +\infty$ for some $c$ where $x_L=e^{y_L}$ and $x_R=e^{y_R}$. This simple argument shows that for any $x_L>0$, $\int_{x_L}^\infty u dx =0$. Therefore the mass of $u$ concentrate in the region $(0, x_L)$. This heuristic argument can be made more rigorous to show the concentration of the stationary distribution and can be generalized to the two-dimensional case with the radial function $f(x)=|x|^2$. Going beyond such specific forms seems extremely difficult.

\subsection{Adding Gradient Information}

One important goal of this paper is to prove that global convergence is possible with an algebraic rate without even approximating the gradient in the algorithm. Another goal is to develop an efficient derivative-free algorithm. Derivative-free methods typically compare different objective function values to find the direction for the next step or to accept a step or not. This is so for
deterministic techniques, for example, the simplex method~\cite{LaReWrWr-SIAM98} and also for stochastic
algorithms, for example, simulated annealing~\cite{kirkpatrick1983optimization}, and consensus-based optimization methods~\cite{carrillo2021consensus,totzeck2021trends}.

\begin{figure}[!htb]
    \centering
    \includegraphics[width = 1.0\textwidth]{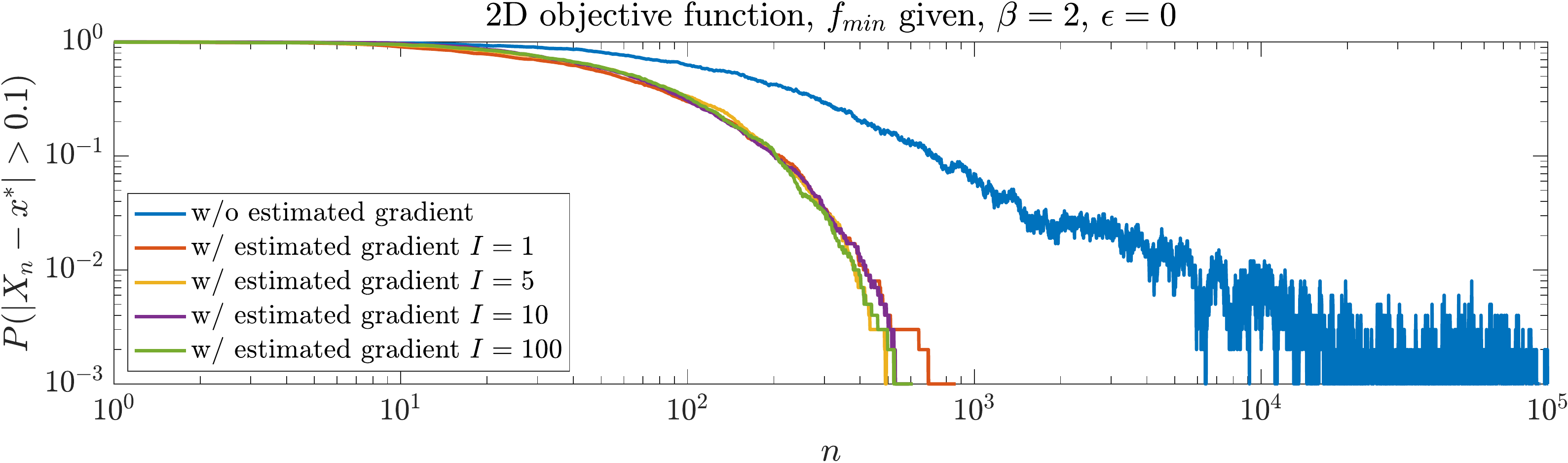}
    \caption{Log-log plots of convergence performance between~\eqref{EQ:Adaptive A} and~\eqref{eq:discrete_SGD}. For both cases, we set $D_n(X_n) = f(X_n)^2$ and $f_{\min}= 0$ is known \emph{a priori}. We also consider different window size $L$ for estimating the gradient in~\eqref{eq:estimate grad}. The statistics are estimated by $10^3$ i.i.d.~runs.}
    \label{fig:2D gradient comp}
\end{figure}

\begin{figure}[!htb]
    \centering
    \includegraphics[width = 1.0\textwidth]{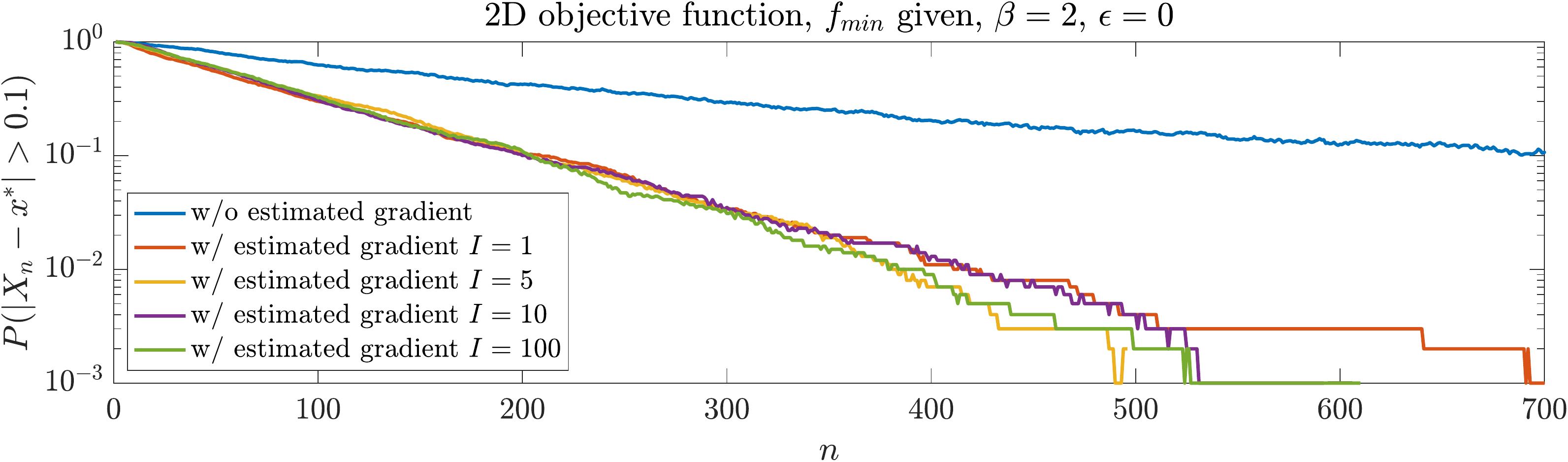}
    \caption{Semilog plots of the same convergence statistics in~\Cref{fig:2D gradient comp} but for $n\leq 700$.}
    \label{fig:2D gradient comp 2}
\end{figure}

Even if the gradient information is not necessary for convergence, adding such information from
objective function values of several steps is also possible here. Without extra computational
cost, the practical performance can be improved. We propose the following simple algorithm. First, we can accelerate the convergence with an approximated gradient based on the secant method as follows
\begin{equation} \label{eq:estimate grad}
     \overline{G}(X_n) = \sum_{i=1}^{I} w_i \frac{f(X_{n-i+1} ) - f(X_{n-i}) }{|X_{n-i+1} - X_{n-i}|^2}\, \left(X_{n-i+1} - X_{n-i}\right),
\end{equation}
where $\sum_{i=1}^I w_i = 1$, and $w_i \geq 0$. For example, we can set the weight $w_i \sim \gamma^i$ for some $0<\gamma < 1$. Using $\overline{G}(X_n)$ in place of the gradient term in a standard stochastic gradient descent scheme, we derive a modified algorithm compared to~\eqref{EQ:Adaptive A}:
\begin{equation} \label{eq:discrete_SGD}
    X_{n+1} = X_n  - \eta_g\, \overline{G}(X_n) + \eta\,\sigma(f(X_n)) \boldsymbol{\zeta}_n,
\end{equation}
where $\eta_g$ is the step size for the gradient term and other symbols follow earlier notations in~\eqref{EQ:Adaptive A}.

We performed simulations using this algorithm with an estimated gradient. In~\Cref{fig:2D gradient comp}, we present result for the case when $f_{\min} = f(\bx_*)$ is known \emph{a priori} and $\sigma(X_n)=\sqrt{2(f(X_n)-f_{\min})^2}$ (that is, the case of $\beta = 2$ and $\eps=0$). We use the weights $w_i \sim \gamma^i$ where $\gamma = 0.5$ and various $I$ values as used in~\eqref{eq:estimate grad}. We compare the convergence performance of descent algorithms based on~\eqref{EQ:Adaptive A} and~\eqref{eq:discrete_SGD}. The statistics are estimated from $10^3$ i.i.d.~runs. It is evident from the log-log plots in~\Cref{fig:2D gradient comp} that the approximated gradient information significantly accelerates the convergence of the stochastic descent algorithm when $n$ is large. The semilog plots in~\Cref{fig:2D gradient comp 2} illustrate the \textit{exponential} convergence when approximated gradients are used in the descent algorithm.

Next, we show an example of full-waveform inversion (FWI). FWI is a nonlinear inverse technique that utilizes the entire wavefield information to estimate the medium properties of the propagating domain. Without loss of generality, the PDE constraint of FWI is the following acoustic wave equation with zero initial condition and non-reflecting boundary conditions.
\begin{equation}\label{eq:FWD}
     \left\{
     \begin{array}{rl}
     & m(\mathbf{x})\frac{\partial^2 u(\mathbf{x},t)}{\partial t^2}- \Laplace u(\mathbf{x},t) = s(\mathbf{x},t),\\
    & u(\mathbf{x}, 0 ) = 0,                \\
    & \frac{\partial u}{\partial t}(\mathbf{x}, 0 ) = 0 . 
     \end{array} \right.
\end{equation}
We set the model parameter $m(\mathbf{x}) = 1/c(\mathbf{x})^2$, where $c(\mathbf{x})$ is the wave velocity, $u(\mathbf{x},t)$ is the forward wavefield, $s(\mathbf{x},t)$ is the wave source. The velocity parameter $m$ is often the target of reconstruction. Equation~\eqref{eq:FWD} is a linear PDE but defines a nonlinear operator $\mathcal F$ that maps $m(\mathbf{x})$ to $u(\mathbf{x},t)$. In FWI, we translate the inverse problem of finding the model parameter $m$ based on the observable seismic data $\{g^{obs}_i\}$ to a constrained optimization problem:
\begin{equation}\label{eq:FWI}
    m^{\ast} = \argmin_{m} f(m),\quad f(m) = \frac{1}{2} \sum_{i=1}^{n_s} \int_\Gamma \int_0^T \| g_i(x,t;m) - g^{obs}_i(x,t)\|^2 dt\,dx\,,
\end{equation}
where $n_s$ is the number of wave sources. For each given source  $s_i(\mathbf{x},t)$ where $1\leq i\leq n_s$,  $g_i(x,t;m) = R\mathcal F(m)$ is the synthetic data  with $R$ being the linear projection operator that extracts the wavefield $u_i$ at the measurement domain $\Gamma$.

We comment that~\eqref{eq:FWI} is a highly-nonconvex optimization problem. We will apply our AdaVar algorithm with an additional approximated gradient component~\eqref{eq:discrete_SGD} to find the global minimizer. First, we parameterize the velocity $c(\mathbf{x})$ to be piecewise-constant and we wish to invert ten unknowns $\{v_i\}_{i=1}^{10}$; see~\Cref{fig:FWI-vel} for an illustration. That is, we search for $X = [v_1,\ldots,v_{10}] \in [1.5,5.5]^{10}\subset \mathbb{R}^{10}$. Thus, the objective function can be denoted as $f(m) = f(v_1,\ldots,v_{10}) = f(X)$. In executing the algorithm~\eqref{eq:discrete_SGD}, we set $\sigma(f(X_n)) = |f(X_n)|^3$, i.e., $\beta = 6 > d/2 = 5$, $\eta = 0.125$, $\eta_g = 0.05$, $\gamma = 0.5$ and $L = 2$, as the hyper-parameters. We consider $f_{\min} = 0$ since this is a data-fitting problem. Since the last layer right above the bottom boundary cannot be accurately recovered due to the non-reflective boundary condition, we assume its velocity is known to be $5$ km/s. We place $8$ sources and $60$ receivers equally distributed on the top boundary. The source consists of two Ricker wavelets of disjoint supports at $15$ Hz peak frequency. The ground truth is $X^* = [4.81,4.77,4.75, 4.83,4.94,5.35,4.67,4.83,5.05,5.18]$.

In~\Cref{fig:FWI-obj,fig:FWI-iter}, we plot the convergence histories of the objective function values and the iterates. In the first $500$ iterations, the update is dominated by noise as the objective function value, and the errors in the iterates fluctuate randomly. Later, the approximated gradient becomes the leading driving force since the objective function decays almost monotonically. The top-layer coefficients, $v_1, \ldots, v_5$ converge to the ground truth as measured in the $\ell^2$ error; see~\Cref{fig:FWI-iter}. The bottom-layer coefficients converge much slower as the objective function is not very sensitive to their changes, and the estimated gradient biases towards sensitive coefficients. We also plot the results using the gradient descent algorithm starting from a homogeneous velocity of $2$ km/s. The iterates get stuck at a local minimum in fewer than $20$ iterations, as we can see from~\Cref{fig:FWI-obj,fig:FWI-iter}.

\begin{figure}[!htb]
\centering
\subfloat[velocity parameterization]{\includegraphics[width = 0.33\textwidth]{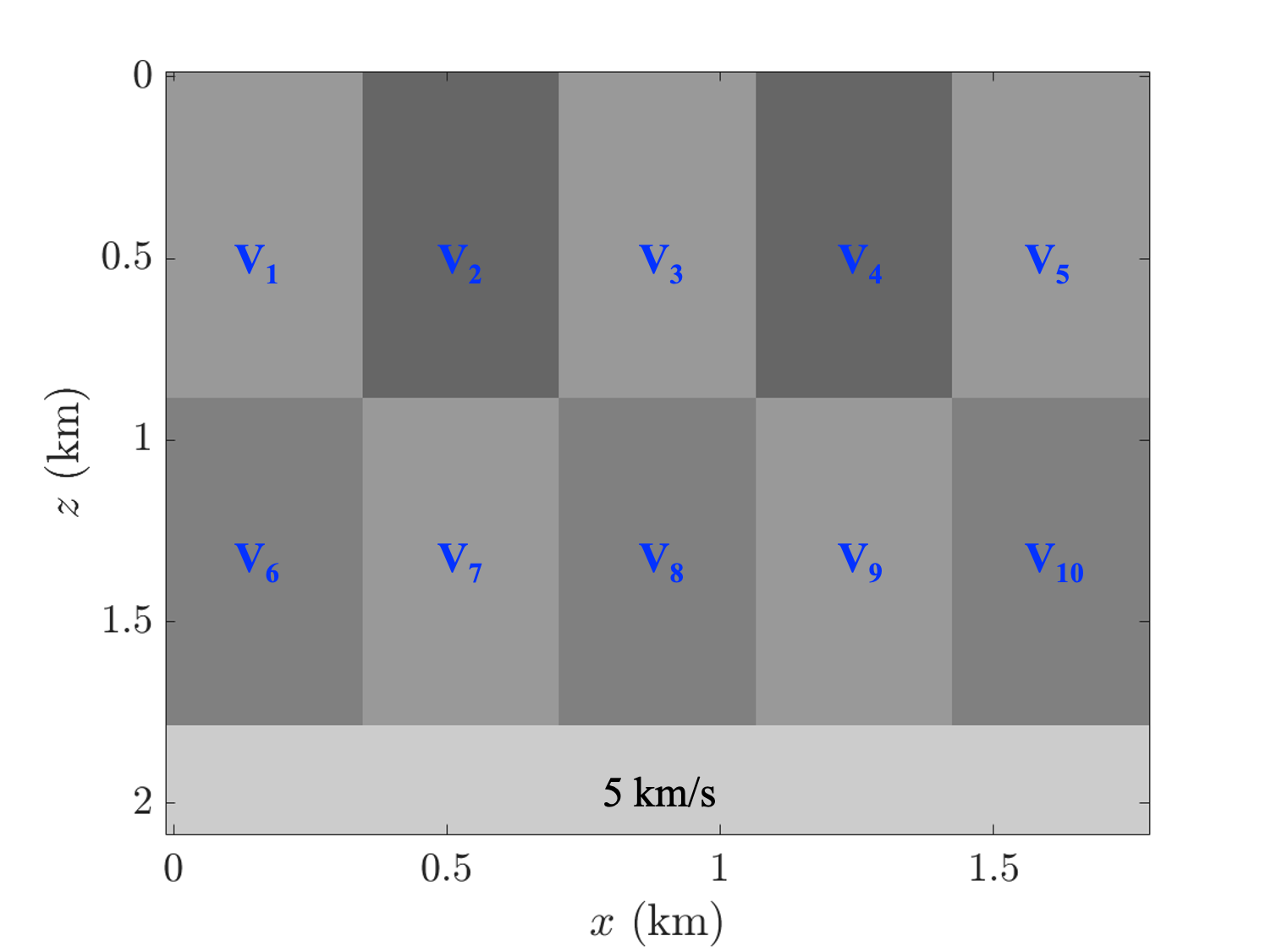}\label{fig:FWI-vel}}
\subfloat[objective function value]{\includegraphics[width = 0.33\textwidth]{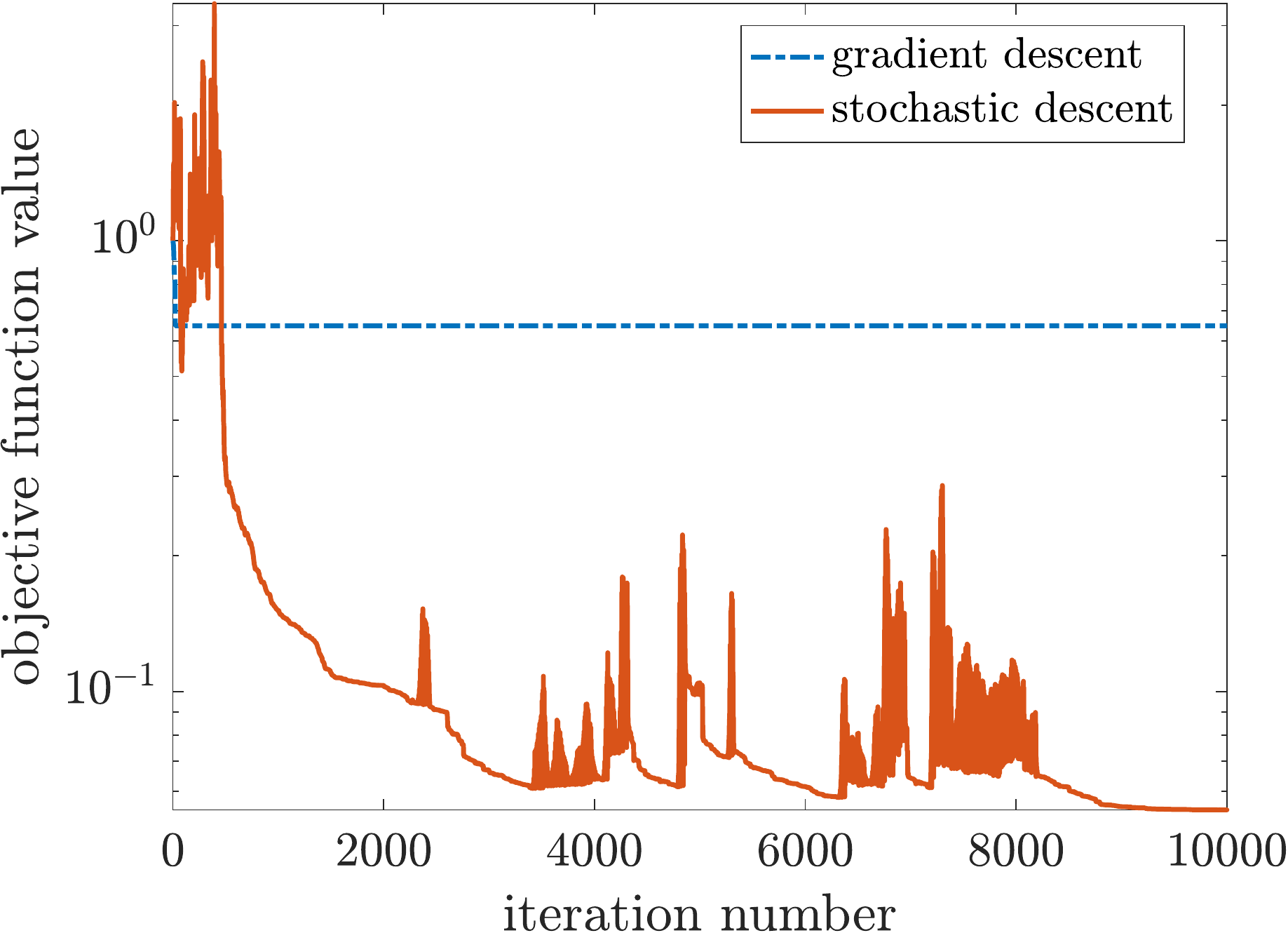}\label{fig:FWI-obj}}
\subfloat[$\ell^2$ error in iterates]{\includegraphics[width = 0.33\textwidth]{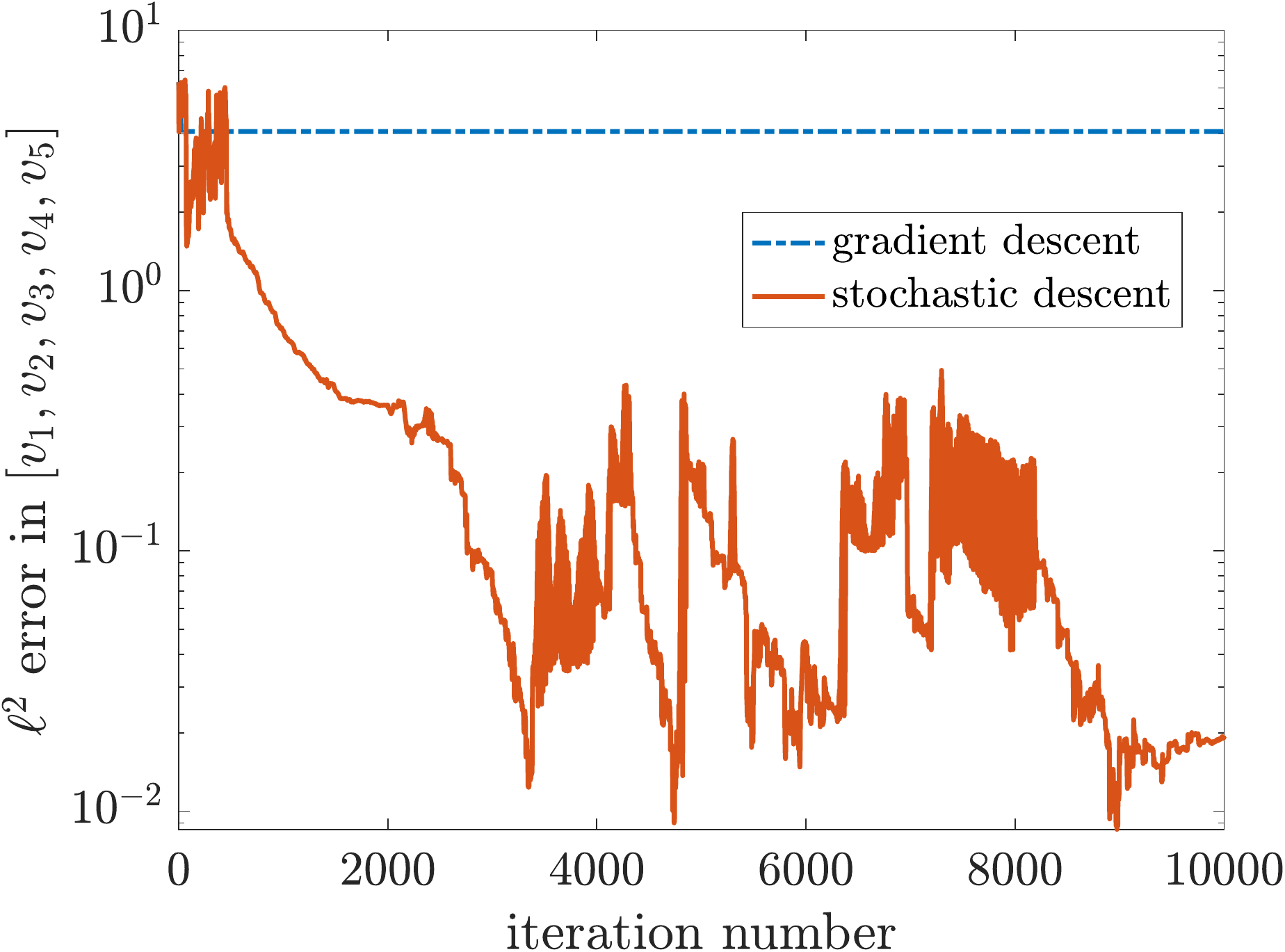}\label{fig:FWI-iter}}
\caption{Global optimization for FWI: the velocity parameterization with $10$ unknowns (left), the objective function value decay (middle), and the convergence history of $[v_1,\ldots,v_5]$ (right) using the 
proposed stochastic algorithm and the standard gradient descent algorithm.\label{fig:FWI-history}}
\end{figure}

This method improves the convergence rate over~\eqref{EQ:Adaptive A} significantly, particularly in higher dimensions, as seen in the numerical experiments above. The original algorithm~\eqref{EQ:Adaptive A} does not suffer from the curse of dimensionality in the same way as in standard quadrature and PDE
methods for which the discretization is done dimension by dimension. The algorithm here still
shows severe degradation in modestly higher dimensions because $\beta$ in~\eqref{EQ:Sigma} depends on $d$ in determining the noise power $\sigma(f)$. The convergence of the classical gradient
descent method is essentially independent of dimensional degradation. Thus, it is natural
to add gradient information  such as~\eqref{eq:discrete_SGD} to have a practical algorithm.

\begin{remark}\label{rem:beta_grad}
Here, we comment that the choice of $\beta$ differs from our earlier discussions when the (approximated) gradient is present. The $\beta$ values that give the best convergence for our derivative-free method are quite large; see~\Cref{thm:main}. With explicitly adding the (approximated) gradient, the stochastic term with a large $\beta$ is then too weak (given the fact that $f(x)-f_{\min}^* \in [0,1]$ in our test cases) to escape a local minimum and overcome the adverse gradient in a reasonable time. This is particularly the case when the objective function value at the local minimum is close to the estimated optimum value $f_{\min}^*$. A smaller $\beta$ naturally implies more noise based on the standard deviation $\sigma \sim |f(x)-f_{\min}^*|^{\beta/2}$ when $|f(x)-f_{\min}^*|\leq 1$, thereby increasing the probability of escape. Without the (approximated) gradient, the condition to ensure convergence in probability is that $\beta > d/2$; see~\Cref{coro:convergence in p}. The same condition does not carry over to the case when the (approximated) gradient is present.
\end{remark}

We can think of our estimated gradient $\overline G$ as a noisy version of the true gradient, i.e., $\overline G\approx \nabla f +\xi_k$. As a result, the iteration~\eqref{eq:discrete_SGD} with a constant $\sigma$ converges, in probability, to the global minimizer of $f$ under the right scaling (which essentially is that $\eta_g \sim 1/k$ and $\eta\sim 1/(k\log \log k)$). This can be shown with a slight modification of the techniques from~\cite{GeMi-SIAM91}, with minor additional assumptions on $f$. 

\section{Revisiting the AdaVar Stochastic Gradient Descent Algorithm in\texorpdfstring{~\cite{EnReYa-arXiv22}}{}}
\label{SEC:compare}

The basic concept of adding a stochastic term in the optimization algorithm with the variance of that term
being state-dependent was already introduced in~\cite{EnReYa-arXiv22}. There are two main differences between our earlier paper~\cite{EnReYa-arXiv22} and this work. One is  obviously that the earlier algorithm explicitly included the gradient, and the main proposal in this paper is derivative-free. In this paper, the step from $X_n$ to $X_{n+1}$ is here taken at a uniformly random angle (due to the isotropic Gaussian noise in~\eqref{EQ:Adaptive A}). The optimization landscape must be explored in a sequence of many steps to find a descent direction, and our analysis in this work can, therefore, not be done in a Markovian way on the discrete level based on worst-case scenarios, which was done in~\cite{EnReYa-arXiv22}. We need the probability distribution of $X_n$ here, and it is, therefore, natural to study the continuum limit in the form of the Fokker--Planck equation~\eqref{EQ:Fokker--Planck}. This is
common in convergence analysis; see, for example, \cite{GeHw-SIAM86}. 

We also comment that the convergence proof of~\cite{EnReYa-arXiv22} will not work without the gradient. The proof of Property One will not be affected~\cite[Sec.~3.1]{EnReYa-arXiv22} since it did not use the gradient explicitly, which, however, is essential in the proof of Property Two~\cite[Sec.~3.2]{EnReYa-arXiv22}. Without the gradient, the worst-case scenario in~\cite{EnReYa-arXiv22} will generate a very high probability for the iterate $X_n$ to escape the set $\Omega_n = \{{\bf x}\in\Omega: f({\bf x}) \leq f_n\}$, which is endowed with small noise variance. The discrete Markovian-style analysis will then not work. We have to use the history of $X_n$-values and the related probability density function to show that such worst-case scenarios have a small probability. 

The other difference is the choice of the adaptive state-dependent noise term $\sigma(f)$. In~\cite{EnReYa-arXiv22}, it was a piecewise-constant step function based on the value $f({\bf x})$: 
\begin{equation}\label{eq:two-stage-sigma}
 \sigma_n(f(X_n)) =    
 \begin{cases}
  \sigma_n^-, & f(X_n) \leq f_n,\\
   \sigma_n^+, & f(X_n) > f_n,
    \end{cases}
\end{equation}
where $f_n$ is a cut-off function decaying in $n$ towards $f(\bx_*)$. This was useful in its simplicity both for the analysis and in producing practical convergence. 
Without gradient information, noisy iterates driven by a constant variance will have a long hitting time to reach a close neighborhood of a global minimum. See the comments above and also in~\cite{EnReYa-arXiv22} for the importance of the gradient in this phase of the algorithm, i.e., Property Two. Using a $\sigma(f)$, which is a strictly monotone function of the objective function value $f(\bx)$, will implicitly exploit gradient information over a sequence of steps throughout the full domain $\Omega$. In this work, we indeed use such a regular monotone function of $|f({\bf x})-f_{\min}^*|$; see~\eqref{EQ:Sigma}. 
The particular choice of a regular variance function in this paper fits nicely into the analysis of the Fokker--Planck equation. From a practical point of view, the advantage of the currently proposed variance is the implicit encoding of the gradient, as remarked on in~\eqref{eq:grad-FK}. The high-dimensional example in the earlier paper~\cite[Sec.~5.1.2]{EnReYa-arXiv22} showed a much faster convergence into the basin of attraction of the global minimum than the result that a uniform sampling would have given. The gradient was a key in guiding the sequence of $X_n$-values closer to the optimum. If a step function-type variance is used in the derivative-free setup, we will only rely on uniform sampling to find the domain close to the optimum, resulting in very slow convergence.

It is natural to ask how a monotone $\sigma(f)$ will do in the gradient descent algorithm of~\cite{EnReYa-arXiv22} and
how the piecewise-constant variance~\eqref{eq:two-stage-sigma} works without gradients in the framework of~\eqref{EQ:Adaptive A}. The numerical examples below will shed light on these questions.

For numerical comparisons between the step function-based variance~\eqref{eq:two-stage-sigma} proposed in~\cite{EnReYa-arXiv22} and the continuous variance~\eqref{EQ:Sigma} proposed in this work, we will first show when the gradient is present (that is, the approximated gradient in~\eqref{eq:estimate grad} is replaced by the real one), how the two strategies perform in global optimization. The test case is the 2D Rastrigin function in~\cite[Eqn.~(5.1)]{EnReYa-arXiv22}. In~\Cref{fig:2D var comp}, we observe that for both cases of  $c=0.01$ and $c=0.05$, iterates driven by the continuous variance (labeled as ``cont-var'') converge much faster than the step function-based variance (labeled as ``bi-var'') in the plot. We applied the same strategy in choosing the hyper-parameters in the SGD algorithm with the step function-based variance~\eqref{eq:two-stage-sigma} as in~\cite{EnReYa-arXiv22}. 

\begin{figure}[!htb]
    \centering
    \includegraphics[width = 1.0\textwidth]{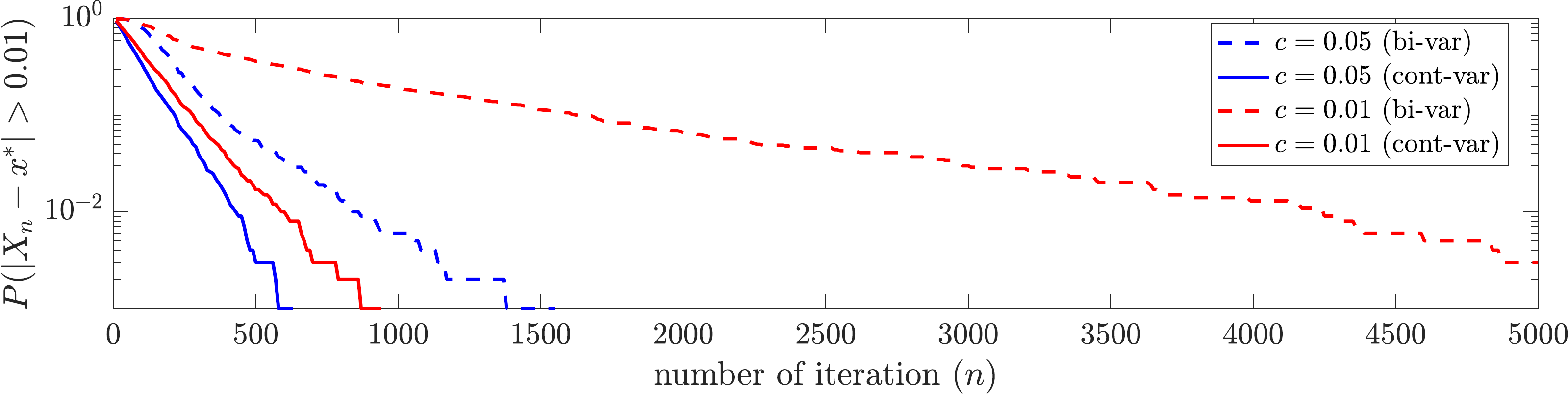}
    \caption{Semilog-y plots of convergence performance for stochastic gradient descent with step function-based variance~\eqref{eq:two-stage-sigma} and the continuous variance~\eqref{EQ:Sigma}. The statistics are estimated by $10^3$ i.i.d.~runs.}
    \label{fig:2D var comp}
\end{figure}

On the other hand, we can remove the gradient component in the SGD algorithm in~\cite{EnReYa-arXiv22}, leaving only the zero-mean noise with the two-stage variance~\eqref{eq:two-stage-sigma} controlling the trajectory of the iterate $X_n$. As mentioned earlier, the same proof~\cite{EnReYa-arXiv22} will not go through. We can also observe the difficulty in convergence from numerical tests. Consider the objective function~\eqref{eq:J} in 2D. The trajectory of one run using the continuous variance was shown earlier in~\Cref{fig:2D epsilon compare 1}. Similarly, we plot the trajectory of one run with the two-stage variance~\eqref{eq:two-stage-sigma} in~\Cref{fig:2D bi-var no grad}. The iterate $X_n$ has been close to the global minimum ${\bf x}_*=[2,2]^\top$ many times in the trajectory history but has also escaped shortly after. As the sublevel set $\Omega_n = \{{\bf x}:f({\bf x}) \leq f_n\}$ shrinks with $f_n$ decreasing, the ``leaving'' probability $\mathbb{P}(X_{n+1} \not\in \Omega_n |X_n \in \Omega_n)$ might be large as the volume of $\Omega_n$ decreases while the ``entering'' probability $\mathbb{P}(X_{n+1} \in \Omega_n |X_n \not\in \Omega_n)$ becomes smaller again due to the decreasing volume of $\Omega_n$. The fact that the ratio between these two conditional probabilities goes to zero was the key in the proof of~\cite{EnReYa-arXiv22}, which was driven by the gradient term. Without the gradient term, the same analysis does not apply.

\begin{figure}[!htb]
    \centering
\includegraphics[height = 5.5cm]{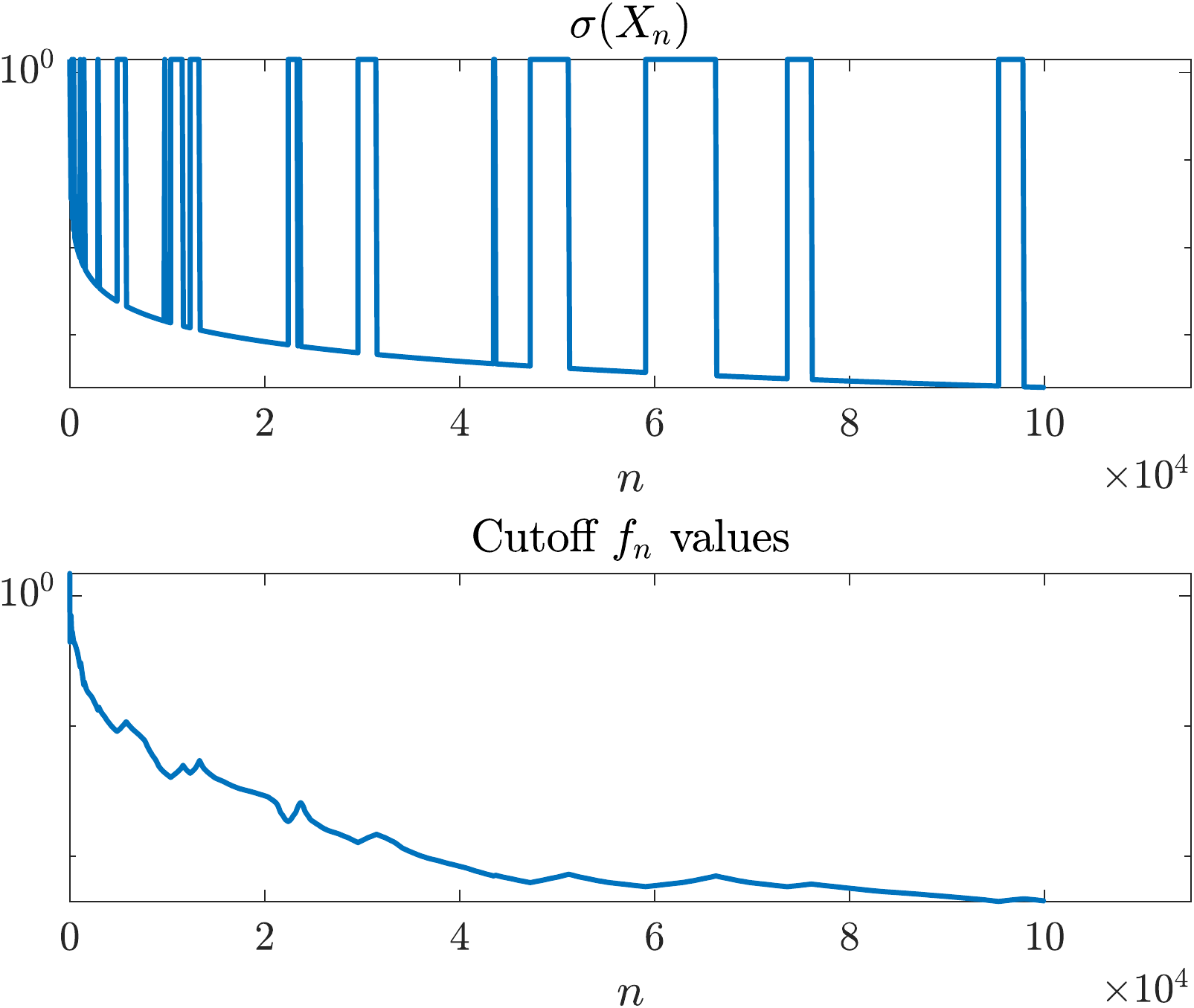}
\hspace{0.5cm}
\includegraphics[height = 5.5cm]{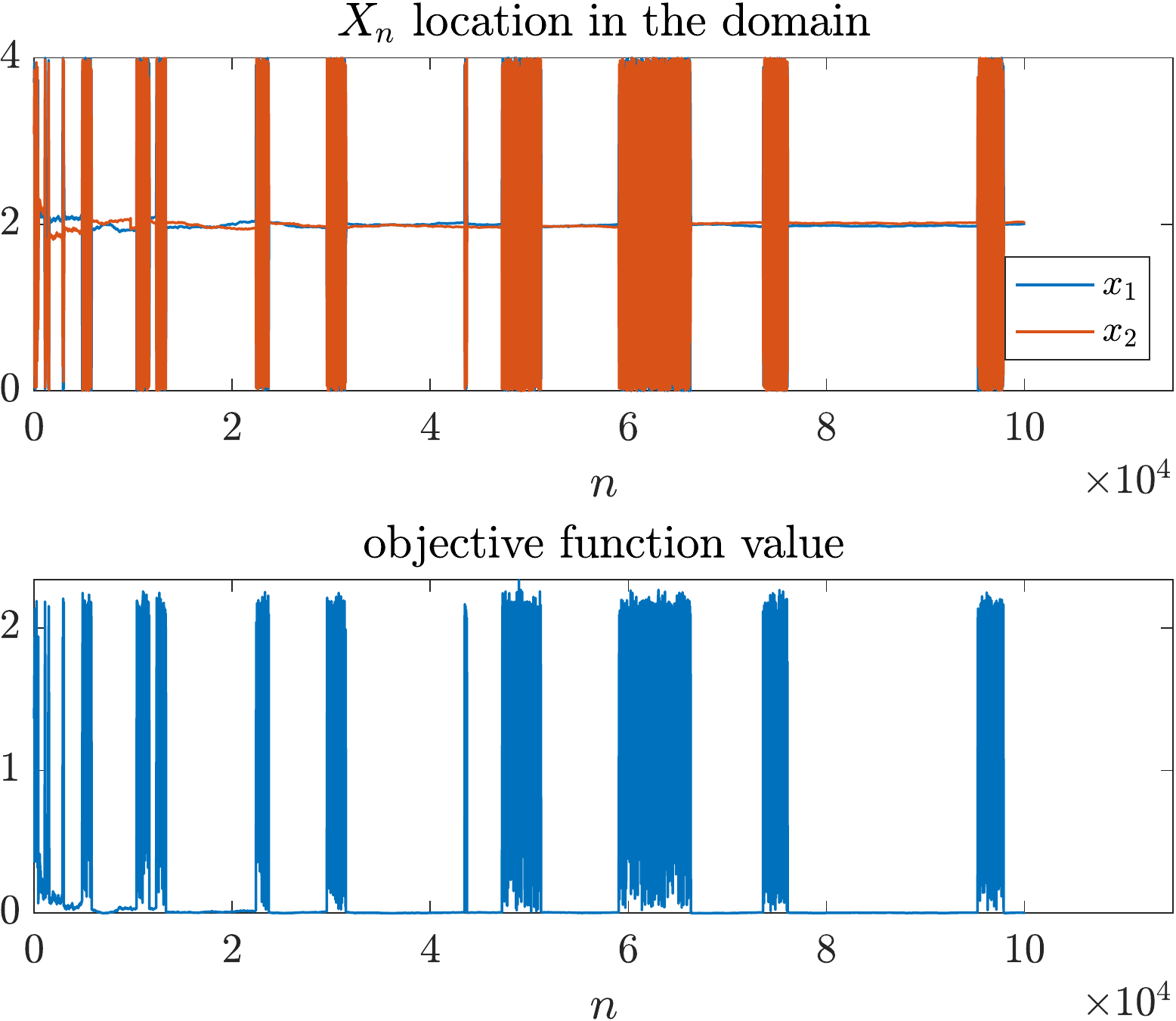}
    \caption{Stochastic descent without gradient driven by the two-stage variance~\eqref{eq:two-stage-sigma} for the 2D objective function~\eqref{eq:J}. Top left: effective variance $\sigma(X_n)$ at the $n$-th iteration; Bottom left: the cut-off value $f_n$ in~\eqref{eq:two-stage-sigma}; Top right: locations of $X_n$; Bottom right: the objective function value $f(X_n)$.}
    \label{fig:2D bi-var no grad}
\end{figure}

\section{Concluding remarks}
\label{SEC:Concl}

We have presented analysis and computational evidence to demonstrate the efficiency of an adaptive variance selection scheme for derivative-free optimization. While our theoretical justification is in the asymptotic regime, numerical simulations with the discrete algorithm show that the method works remarkably well in more challenging settings, for example, when the true value of the global minimum of the objective function is unknown.
 
The main difference between this contribution and other derivative-free methods is the rigorous analysis of global convergence with the algebraic rate, even in the case of no explicit gradient approximation. There are also several differences between the current work and our previous paper~\cite{EnReYa-arXiv22} in which a discrete version was studied. First, in paper~\cite{EnReYa-arXiv22}, the proof of Property Two does not work without gradient information. Second, the probability distribution is needed in this work, and it depends on the objective function value of the iterate. The proof in~\cite{EnReYa-arXiv22} is instead based on the discrete algorithm. Several interesting theoretical issues remain to be addressed, including the convergence of the algorithm in the case of $\eps = 0$, having $\Omega = \mathbb{R}^d$, and including the estimation of $f_{\min}$ in the analysis. There are also more practical issues as, for example, the best choice of gradient approximation and involving parallel sequences of $X_n$ in the optimization. We leave those to future works.






\section*{Acknowledgments}

We are grateful for the valuable discussions with Professor Linan Chen (McGill University) and Professor Panagiotis E.~Souganidis (University of Chicago) for constructive discussions. 

This work is partially supported by the National Science Foundation through grants DMS-2208504 (BE), DMS-1913309 (KR), DMS-1937254 (KR), and DMS-1913129 (YY). YY acknowledges support from Dr.~Max R\"ossler, the Walter Haefner Foundation, and the ETH Z\"urich Foundation. 

\appendix

\section{Weighted Poincar\'e inequality}
\label{SEC:WPI}
One key component in our analysis is the weighted Poincar\'e inequality with a given weight function $w(\bx):\Omega \mapsto [0, \infty)$. To increase the readability of our proof, we recall here the inequality. The material here is standard and can be found in the references cited. For a general weight function $w(\bx)$,  we first introduce the Muckenhoupt
$A_p$ weights.
\begin{definition}[$A_p$ weights]\label{DEF:Ap}
For a fixed $1 < p < \infty$, we say that a weight function $w:\R^d \mapsto [0, \infty)$ belongs to the class $A_p$ if $w$ is locally integrable, and for all cubes $Q\subset \R^d$, we have
\begin{equation}
[w]_p : = \sup_{Q\subset \R^d} \left({\frac {1}{V_Q }}\int _{Q}w (\bx)\,d\bx\right)\left({\frac {1}{V_Q}}\int _{Q}w (\bx)^{-{\frac {q}{p}}}\,d\bx\right)^{\frac {p}{q}} <\infty \,,
\end{equation}
where $q$ is a real number such that $\frac{1}{p} + \frac{1}{q} = 1$, and $V_Q$ is the volume of the cube $Q$.
\end{definition}

The following weighted Poincar\'e inequality for weights in the $A_p$ class can be found in~\cite[Proposition~11.7]{ perez2019degenerate}.
\begin{theorem}[Weighted Poincar\'e inequality~\cite{perez2019degenerate}]\label{thm:wpi}
Let $w$ be an $A_p$ weight function and $f(\bx)$ a Lipschitz function. Then the following weighted Poincar\'e inequality holds for the hypercube $\Omega \subset \R^d$:
\begin{equation} \label{eq:wpi}
\frac{1}{w(\Omega)} \int_\Omega |f-f_{\Omega,w}|^p  w\, d\bx \leq\frac{ 2^p }{w(\Omega)} \int_\Omega |f-f_{\Omega}|^p  w\, dx \leq  \frac{C_d^p\, \ell_\Omega^p\, [w]_p\, }{w(\Omega)} \int_\Omega |\nabla f|^p w\, d\bx\,,
\end{equation} 
where $w(\Omega) = \dint_\Omega w(\bx) d\bx$,  $f_{\Omega,w} = \frac{1}{w(\Omega)} \dint_\Omega f(\bx) w(\bx) d\bx$, $f_{\Omega} = \frac{1}{V_\Omega} \dint_\Omega f(\bx) d\bx$, $\ell_\Omega$ is the side length of the cube $\Omega$, and $C_d$ is a dimensional constant.  
\end{theorem}

There have been many results on the weighted Poincar\'e inequality~\cite{fabes1982local,heinonen2018nonlinear}. The paper by P\'erez and Rela~\cite{perez2019degenerate} improved some of the classical results and produced a quantitative control of the Poincar\'e constant (see~\eqref{eq:wpi}) in the inequality, which is crucial for the analysis of our algorithm. We refer interested readers to~\cite{heinonen2018nonlinear,perez2019degenerate} for more general weighted Poincar\'e and Poincar\'e--Sobolev inequalities in various settings.

\begin{remark}
The definition of the $A_p$ class allows one to consider degenerate and singular weights. For example, let $w(\bx) = |\bx|^{\eta}$, $\bx\in\bbR^d$. Then $w\in A_p$ if and only if $-d< \eta < d(p-1)$.
\end{remark}

{\small
\bibliography{BIB-REN2}
\bibliographystyle{siam}
}

\end{document}